\documentclass[12pt, twoside]{article}
\usepackage{amsmath,amsthm,amssymb}
\usepackage{times}
\usepackage{enumerate}

\def\titlerunning#1{\gdef\titrun{#1}}
\makeatletter
\def\author#1{\gdef\autrun{\def\and{\unskip, }#1}\gdef\@author{#1}}
\def\address#1{{\def\and{\\\hspace*{18pt}}\renewcommand{\thefootnote}{}%
\footnote {#1}}%
\markboth{\autrun}{\titrun}}
\makeatother
\def\email#1{e-mail: #1}
\def\subjclass#1{{\renewcommand{\thefootnote}{}%
\footnote{\emph{Mathematics Subject Classification (2010):} #1}}}
\def\keywords#1{\par\medskip
\noindent\textbf{Keywords.} #1}

\def\eps{\epsilon}%
\def\tensor{\,\raise2pt\hbox{${}_{\otimes}$}\,}
\def\fdg{\,:\,}
\def\ptl{\partial}
\def\rest#1{\raise-2pt\hbox{${\lfloor_{#1}}$}}

\def\mbo#1{\boldsymbol{#1}}

\def\ip#1#2{\langle#1,#2\rangle}

\def\hatt#1{\widehat #1{}}
\def\olin#1{\overline{#1}{}}

\def\grad{{\nabla}}
\def \bgrad{\bar{\nabla}}
\newcommand{\leftexp}[2]{{\vphantom{#2}}^{#1}{#2}}

\def\halb{\frac{1}{2}}

\def \a{\alpha}
\def \b {\beta}

\def \bric{\olin{\text{Ric}}}

\def \brie {\olin{\text{Riem}}}
\def \rie {\text{Riem}}

\newtheorem{theorem}{Theorem}[section]
\newtheorem{lemma}[theorem]{Lemma}

\newcommand{\ba}{\begin{array}}
\newcommand{\ea}{\end{array}}
\newcommand{\bea}{\begin{eqnarray}}
\newcommand{\eea}{\end{eqnarray}}
\newcommand{\bee}{\begin{eqnarray*}}
\newcommand{\eee}{\end{eqnarray*}}

\renewcommand{\a}{\alpha}
\renewcommand{\b}{\beta}

\usepackage{color}

\newcommand{\green}[1]{{\color{green}#1}}

\newcounter{mnotecount}[section]

\renewcommand{\themnotecount}{\thesection.\arabic{mnotecount}}

\newcounter{mymnotecount}[section]
\renewcommand{\themymnotecount}{\thesection.\arabic{mymnotecount}}
\newcommand{\mymnote}[1]{\protect{\stepcounter{mymnotecount}}${\raisebox{0.5\baselineskip}[0pt]{\makebox[0pt][c]{\color{green}{\tiny\em$\bullet$\themnotecount}}}}$\marginpar{\raggedright\tiny\em$\!\bullet$\themymnotecount:

\green{#1}}\ignorespaces}

\renewcommand{\mymnote}[1]{}
\allowdisplaybreaks
\begin{document}
	
	\baselineskip=17pt
	

	\titlerunning{Curvature Propagation}
	\title{Curvature Propagation  for the 3+1 Dimensional U(1) Symmetric Einstein Spacetimes}
	
	\author{Nishanth Gudapati}
	
	\date{}
	
	\maketitle

	\address{N. Gudapati: Department of Mathematics, Clark University, 950 Main Street
		Worcester, MA 01610; \email{ngudapati@clarku.edu}}

	
	
	\begin{abstract}
		As it is well known, the global structure of the Einstein equations for general relativity in the context of the initial value problem, is a difficult and intricate mathematical problem. Therefore, any additional structure in their formulation is useful as a tool for studying the global behaviour of the initial value problem of the Einstein equations. In our previous works, we have used the additional structure provided by the dimensional reduction to $2+1$ dimensional Einstein-wave map system. In this work, we shall focus on yet another structure in the Einstein equations for the U(1) symmetric spacetimes, namely the analogy with the Yang-Mills theory and reconcile with the dimensionally reduced field equations.

	\end{abstract}
\section{Background and Introduction}
The subject of this article is the existence problem of the Einstein equations for general relativity: 

\begin{align} \label{ein-orig}
    \olin{R}_{\mu \nu} =0, \quad (\bar{M}, \bar{g})
\end{align}

defined on a 4-dimensional Lorentzian spacetime $(\bar{M}, \bar{g}).$ The existence problem of the Einstein equations is formulated in terms of the initial value problem or the Cauchy problem of \eqref{ein-orig}. For the Cauchy problem of the Einstein equations, the constraints play an important role: 
\begin{subequations}
\begin{align}
    \bar{H} =& R_{\bar{q}} +  \bar{q}^{ij} \bar{k}_{ij} - \Vert \bar{k} \Vert^2, \\
 \bar{H}_i=&   \leftexp{(\bar{q})}{\grad}^j (\bar{k}_{ij} - \bar{q}_{ij} \textnormal{tr} \bar{k}), \quad (\olin{\Sigma}, \bar{q}) 
\end{align}
\end{subequations}
defined on the Riemannian manifold $(\olin{\Sigma}, \bar{q})$ after a $3+1$ ADM decomposition of the Lorentzian spacetime 
$(\bar{M}, \bar{g}) =\fdg (\olin{\Sigma}, \bar{q}) \times \mathbb{R}.$ It is well known that the Einstein equations are diffeomorphism invariant, as a consequence of the Bianchi identities defined on $(\bar{M}, \bar{g})$. In other words, we have the choice of gauge at our discretion to study the initial value problem of the Einstein equations \eqref{ein-orig}. From the classic result of Yvonne-Choquet Bruhat, if we choose the harmonic gauge $\{ x^\a, \a=0, 1, 2, 3 \}$ such that 

\begin{align}
\square_{\bar{g}} x^\a =0, (\bar{M}, \bar{g}) \quad \a = 0, 1, 2, 3
\intertext{where the covariant d'Alembertian is}
\square_{\bar{g}} = \frac{1}{\sqrt{-\bar{g}}} \ptl_\mu (\bar{g}^{\mu \nu} \ptl_\nu)
\end{align}
the Einstein equations can be reduced to a strictly hyperbolic partial differential equations for the metric $\bar{g}:$

\begin{align}
    \square_{\bar{g}} \bar{g} = N(\bar{g}, \ptl \bar{g}).
\end{align}
The advantage of the harmonic gauge is that, suppose the constraint equations on the initial data $(\olin{\Sigma}_0, \bar{q}_0)$ 

\begin{subequations}
\begin{align}
    \bar{H} = R_{\bar{q}} +  \bar{q}^{ij} \bar{k}_{ij} - \Vert \bar{k} \Vert^2=0, \\
    \leftexp{(\bar{q})}{\grad}^j (\bar{k}_{ij} - \bar{q}_{ij} \textnormal{tr} \bar{k}) =0, \quad (\olin{\Sigma}_0, \bar{q}_0) 
\end{align}
\end{subequations}

are satisfied, the constraints are automatically propagated by virtue of the propagation of the harmonic gauge condition. Therefore, the entire Einstein equations in harmonic gauge are strictly hyperbolic partial differential equations in the harmonic gauge. This result was used by Choquet-Bruhat and Geroch \cite{ycb_classic} to establish a classic, seminal existence result for the Einstein equations for general relativity. This result of major significance allows us to conclude that the Einstein equations for general relativity are amenable to the initial value problem and that questions such as local and global well-posedness are applicable.

\begin{theorem}
Suppose the constraint equations of the Einstein equations are satisfied on the initial data $(\olin{\Sigma}_0, \bar{q}_0)$, then there exists  development of $(\olin{\Sigma}_0, \bar{q}_0)$ which is an extension of every other development of $(\olin{\Sigma}_0, \bar{q}_0)$. This development is unique up to an isometry. 
\end{theorem}

This seminal theorem allows us to discuss the globally hyperbolic, regular maximal future development (vis-a-vis past development)  of a Riemannian Cauchy hypersurface, thus placing the Einstein equations in a firm footing with respect to the Cauchy problem. However, a few important geometric properties of the globally hyperbolic maximal development are not explained by this classic result. For instance, the completeness of causal geodesics of the maximal development, formation of trapped surfaces, black holes or Cauchy horizons - the latter cases typically arise when the former fails. There are several concrete examples of solutions of the Einstein equations\footnote{see e.g., \cite{DC_09_book, Rod_Klain_2012} where the formation of trapped surfaces was established even though the initial data does not admit one}. In order to explain and discern these phenomena, we need stronger estimates and perform further analysis on the behaviour of the Einstein equations. 

In our \cite{diss, AGS_15}, we initiated a program to study the global behaviour of translationally symmetric $U(1)$ symmetric Einstein equations. As we already discussed, the Einstein equations are a system of highly complex system of partial differential equations and are notoriously difficult to analyze globally, especially for large data (this corresponds to the `strong field' regime in physics terminology). Faced with such a scenario, we a search for any additional structure provided by the Einstein equations to improve our understanding of their global behaviour is worthwhile. The advantage of the $U(1)$ symmetric Einstein equations is that they admit  a dimensional reduction to the $2+1$ dimensional Einstein-wave map system: 

\begin{subequations}
    \begin{align}
        E_{\a \b} = R_{ \a \b} - \halb g_{\a \b} R_g=& T_{\a \b} \\
        \square_g U^i + \leftexp{(h)}{\Gamma}^i_{jk} g^{\a \b} \ptl_\a U^j \ptl_b U^k =&0, \quad (M, g)
    \end{align}
\end{subequations}
where the target for the wave map field is the hyperbolic plane $(\mathbb{H}^2, h).$ $T_{\a \b}$

\begin{align}
    T_{\a \b} \fdg = \ip{\ptl_\a U }{\ptl_\b U}_h - \halb g_{\a \b} \ip{\ptl_\sigma U}{ \ptl^\sigma U}_h
\end{align}

is the energy momentum tensor and $\square_g$ is now the covariant d'Alembertian defined on $(M, g)$. This dimensional reduction procedure can also be analogously be performed for $3+1$ dimensional Einstein-Maxwell equations for $U(1)$ symmetric Lorentzian manifolds, in which the target manifold of the wave map $U \fdg (M, g) \to (N, h)$ is the complex hyperbolic plane $\mathbb{H}^2_{\mathbb{C}}.$ The advantage of the dimensional reduction procedure is that we are working in the critical dimension for wave maps \footnote{note that wave maps are energy critical in 2+1 dimensions (see e.g., \cite{chris_tah1, jal_tah, tao_all,  krieg_schlag_ccwm, sterb_tata_long} )}. The essence of our program is intimately tied to wave maps because it can be rigorously proved that the wave maps represent the true dynamical degrees of freedom of the Einstein wave map system. 

Therefore, powerful PDE techniques come into play that are expected to allow us to get a glimpse of the global behaviour of the Einstein equations for large data. This problem is at a fortuitous crossroads of tractability and relevance, being just at one symmetry level below the general Einstein equations for general relativity. Keeping in mind all these ideas, we had proposed a program to study the global existence problem of the Einstein equations based on the following two steps: 

\begin{itemize}
    \item \textbf{(C1) Non-concentration of Energy:} The energy on a spacelike surface inside the past null cone of any point converges to zero as one approaches the vertex of the cone 
    \item \textbf{(C2) Global Existence for Small Data:} For arbitrarily small initial energy, global existence holds i.e., the maximal Cauchy development of the 2+1 Einstein wave map system is causal geodesically complete 
\end{itemize}

The step \textbf{(C1)} rules out the focusing behaviour  of the system, especially away from the mantel of the past null cone of the point of hypothetical singularity. Due to the critical dimension of wave maps and the associated invariance of energy, we can resclale the region inside the past null cone of the point of hypothetical singularity such that the energy inside that region is less than the energy needed for the existence result for small data. This allows us to jack up the regularity insider the past null cone and further allows us to use the continuity argument to smoothly extend the results past the point of first hypothetical singularity. In other words, due to the scaling symmetries of the problem, the non-concentration of energy result reduces the large-data global existence problem to the small data global existence problem, which is relatively more tractable. This is a powerful and efficacious tool from a PDE technique point of view.  

This program was successfully implemented for the equivariant Einstein wave map system\cite{diss, AGS_15}. The successful implementation of the two steps \textbf{(C1)} and \textbf{(C2)} discussed above establishes the global existence, i.e., global propagation of regularity of the initial data  of the 2+1 Einstein-wave map equations (and consequently of the corresponding $3+1$ dimensional $U(1)$ Einstein equations) but does not provide a quantitative characterization of the global asymptotic behaviour of the 2+1 Einstein wave map equations. In this connection, in the works \cite{BN_17, NG_22_1} we studied the scattering behaviour of the system. 

In the aforementioned discussion of the existence problem of the Einstein equations, we considered the spacetime metric as the main object of study. In our formulation of the initial value problem of the Einstein equations, arguably, an equally important geometric quantity is the curvature of the spacetime, in order to obtain the comprehensive characterization of the maximal development of the initial data. It is in this context that the Cartan formalism comes into the picture. 

As we already alluded to, the Einstein equations go hand in hand with the Bianchi identities: 
\begin{align}
    \bgrad_\mu \brie_{\a \b \gamma 
    \delta} + \bgrad_\gamma \brie_{\a \b \delta \mu} + \bgrad_{\delta} \brie_{\a \b \mu \gamma} =0, \quad (\bar{M}, \bar{g})
\end{align}
which results in a wave equation satisfied by the Riemann curvature tensor 
\begin{align}
    \square_{\bar{g}} \brie_{\a \b \gamma \delta} =& 2 \brie_{\a \mu \delta \nu} \brie^{\,\, \mu\,\,\nu}_{\b \,\,\gamma} -2 \brie_{\a \mu \gamma \nu} \brie^{\,\,\mu \,\,\nu}_{\b \,\,\delta} \notag\\
    & \quad +\brie_{\a \b \mu \nu} \brie_{\gamma \delta}^{\quad \mu \nu}, \quad (\bar{M}, \bar{g}).
\end{align}

This equation is referred to as the Penrose wave equation. This equation can schematically be written as 

\begin{align}\label{penrose-intro}
    \square_{\bar{g}} F_{\mu \nu} +  \brie_{\mu \nu}^{\quad \a \b} F_{\a \b} = N_{\mu \nu} (\brie), \quad (\bar{M}, \bar{g}),
\end{align}
where $\displaystyle N_{\mu \nu} (\brie)$ is the nonlinearity depending on the curvature tensor $\brie$ and the gravitatinal curvature  $F_{\mu \nu}$, modeled after the Yang-Mills curvature, is the short hand for the Riemann curvature $\brie^{\hatt{a}}_{\quad \hatt{b} \mu \nu}$ in the Cartan formalism.  
The equation \eqref{penrose-intro} is the main equation of study in our approach. 

In our study of the Penrose wave equation in the Cartan formalism, we aim to benefit from the `representation' integral formula of the equation as well as from the apriori estimates based on `energy' methods. In this context, the following  remark is pertinent. The Einstein equations, transformed using the Cartan formalism, result in a Yang-Mills theory whose gauge group is the non-compact Lorentz group $SO(3,1).$ Consequently, it is important to note that the standard Yang-Mills stress-energy tensor vanishes identically for our problem. In this regard, we aim to use the Bel-Robinson tensor instead, for the `energy' methods.   

\subsection{The Wave Equation}
An integral formula for the wave equation will play an important role in our analysis of the wave equation for the curvature. In this regard, we aim to use the framework developed by Friedlander and further refined by Moncrief
\cite{Moncrief_2005, Moncrief_2014} (see also a related independent work \cite{Klainerman_Rodnianski_07}). These 'integral formula' based ideas are inspired from the classic works of Eardley-Moncrief \cite{Eardley_Moncrief_I,Eardley_Moncrief_II} on global existence of Yang-Mills equations in $3+1$ dimensions.   To put our framework into context, let us briefly discuss the wave operator in the Minkowski space. This will allow us to compare with the framework on a curved background spacetime $(\bar{M}, \bar{g} )$ later on. We pay special attention to the null geometry and the propagation of the solution of the wave equation close the characteristics of the equation. Consider the Cauchy problem of the wave equation, 
\begin{subequations}
\begin{align}
    \square\,  u =&0, \quad \mathbb{R}^{3+1} \\
    u(0, x) =&u_0, \quad \mathbb{R}^{3} \\
    \ptl_t u(0, x) =& u_1, \quad \mathbb{R}^{3}
\end{align}
\end{subequations}

The Poisson-Kirchoff representation formula is given by
\begin{align} \label{Poisson-Kirchoff}
u(t, x) =& \frac{1}{4 \pi} \ptl_t \left( t \int_{\Vert  x' \Vert=1} u_0 (x + tx') d \omega \right) + \frac{t}{ 4\pi} \int_{\Vert x'\Vert =1} u_1 (x + tx') d \omega \\
=& \frac{1}{4 \pi t^2} \int_{\ptl B(x, t)} \Big( (t + u_1 (x') + u_0(y)) + (x-x') \cdot \grad_{x'} u_0 (x') \Big) d \sigma(x')
\end{align}
The Poisson-Kirchoff formula is based on the spherical means 
\begin{align}
 \displaystyle M_f (x, r) \fdg =& \frac{1}{\vert \mathbb{S}^3(x, r) \vert} \int_{\mathbb{S}^3 (x, r)} f \bar{\mu}_{\mathbb{S}^3}
 \intertext{equivalently}
 =& \frac{1}{4 \pi} \int_{\mathbb{S}^2} f(x + r x') d \sigma(x').
 \end{align}
and reduction of the wave operator into a 1D d'Alembertian. The strong Huygen's principle holding can be seen especially in this derivation. The Kirschoff formula can be expressed in terms of the spherical means $M_f$ is as follows: 

\begin{align}
    u(t,x) = \ptl_t ( t M_{u_0} (x, t)) + t M _{u_1} (x, t).
\end{align}

In the derivation the following result is also used on the spherical means $\displaystyle \lim_{r \to 0} M_f (x, r) = f(x), x \in \mathbb{R}^3$ and that the Laplacian on $M_f$ is equivalent to the radial Laplacian $ \displaystyle \frac{\ptl^2}{ \ptl r^2} + \frac{2}{r} \ptl_r,$ which reduces the wave equatin to flat space wave operator in 1D, which in turn can be solved explicitly using the d'Alembert formula. Furthermore, for the nonlinear wave equation, the wave equation
\begin{subequations}
\begin{align}
    \square\,  u =&N (u), \quad \mathbb{R}^{3+1} \\
    u(0, x) =&u_0, \quad \mathbb{R}^{3} \\
    \ptl_t u(0, x) =& u_1, \quad \mathbb{R}^{3}
\end{align}
\end{subequations}

can be solved using the decomposition (Duhamel) 
\begin{align}
    u = u_L + u_N
\end{align}

where the functions $u_L \fdg \mathbb{R}^{3+1} \to \mathbb{R}$ and $u_N \fdg \mathbb{R}^3 \to \mathbb{R}$ are the solutions of the corresponding homogeneous (linear) and inhomegeneous (nonlinear) respectively i.e., $u_L$ is a solution of

\begin{subequations}
\begin{align}
    \square\,  u_L =&0, \quad \mathbb{R}^{3+1} \\
    u_L(0, x) =&u_0, \quad \mathbb{R}^{3} \\
    \ptl_t u_L(0, x) =& u_1, \quad \mathbb{R}^{3}
\end{align}
\end{subequations}
 with the initial data identical to the original non-linear wave equation, and as can be seen below, $u_N$ is a solution of the non-linear wave equation with the nonlinearity identical to that of our original wave equation and vanishing initial data: 
 \begin{subequations}
\begin{align}
    \square\,  u_N =&N, \quad \mathbb{R}^{3+1} \\
    u_N(0, x) =&0, \quad \mathbb{R}^{3} \\
    \ptl_t u_N(0, x) =& 0, \quad \mathbb{R}^{3}.
\end{align}
\end{subequations}

Now then, let us define the evolution operator $S_f (x, t)$ for a function that satisfies 

\begin{subequations} \label{evo-operator}
\begin{align}
    \square\,  u =&0, \quad \mathbb{R}^{3+1} \\
    u(0, x) =&0, \quad \mathbb{R}^{3} \\
    \ptl_t u(0, x) =& f, \quad \mathbb{R}^{3}
\end{align}
\end{subequations}

from which it follows, from Kirchoff formula, that, 
\begin{align}
    u (x, t)= S_{u_1} (x, t) + \ptl_t S_{u_0} (x, t) + \int^t_0 S_{N} (x, t- t') dt'
\end{align}
where the latter term can be interpreted as the infinite sum of the operators satisfying \eqref{evo-operator}. It should be noted that the source evolution operator can also be represented as 

\begin{align}
S_f = f \frac{\sin \sqrt{\Delta} t}{ \sqrt{\Delta}}
 \intertext{and has the explicit structure}
    S_f = \frac{1}{4 \pi t} \int_{\mathbb{S} (x, t)} f(x') \bar{\mu}_{\mathbb{S} (x' t)}.
\end{align}
Subsequently, we get the explicit formula, 

\begin{align}
    u (x, t) =& \frac{1}{4 \pi t^2} \int_{\mathbb{S}^2 (x, t)} (t u_1(x') + u_0 (x') + \grad u_0 (x-x')) \bar{\mu}_{\mathbb{S}^2(x', t)} \notag\\
    & \frac{1}{ 4\pi} \int^t_0 \frac{1}{t-t'} \int_{\mathbb{S}^2 (x, t)} N (x', t') \bar{\mu}_{\mathbb{S}^2 (x')} d t'.
\end{align}

Now then, in the dimensional reduction approach, we need to analyze the wave operator in 2D. Analogous to the 3D case, we have a Poisson-Kirchoff type formula after using the Hadamard method of descent: 

\begin{align}
    u(x, t) =& \frac{1}{2 \pi} \int_{\mathbb{D} (0, 1)} \frac{u_0 (x + tx') } {\sqrt{1- \Vert x' \Vert^2}} dx' + \frac{t}{ 2\pi} \int_{\mathbb{D} (0, 1)} \frac{\grad u_0 (x + tx')}{ \sqrt{1- \Vert x'\Vert^2}} x' dx' \notag\\
    & + \frac{t}{ 2\pi} \int_{\mathbb{D} (0,1)} \frac{u_1 (x + tx')}{1-\Vert x' \Vert^2} dx'
\end{align}

Now then, using similar analysis, for the non-linear wave equation we have, based on the Duhamel's principle, 

\begin{align}
    u(x, t) =& \frac{1}{2 \pi} \int_{\mathbb{D} (0, 1)} \frac{u_0 (x + tx') } {\sqrt{1- \Vert x' \Vert^2}} dx' + \frac{t}{ 2\pi} \int_{\mathbb{D} (0, 1)} \frac{\grad u_0 (x + tx')}{ \sqrt{1- \Vert x'\Vert^2}} x' dx' \notag\\
    & + \frac{t}{ 2\pi} \int_{\mathbb{D} (0,1)} \frac{u_1 (x + tx')}{1-\Vert x' \Vert^2} dx'
    + \frac{1}{ 2\pi} \int^t_0 \int_{\mathbb{D} (x, t-t')} \frac{N (x',t')}{ \sqrt{t^2 - \Vert x-x' \Vert^2}} dx' dt'
\end{align}
which is obtained by the fact that, if the initial data independent of the component $x^3,$ the solution $u (t, x)$ is also independent of $x^3.$ Let us make a further comment about the Huygens' principle. In view of the fact that an integral over a sphere reduces to an integral over a disk in the Poisson-Kirchoff formula, we get a `bulk' integral over a disk. As a consequence, the integral equation of the solution of the wave equation does not obey the strong Huygens' principle. This discussion will be relevant to isolate the reasons for the Huygens' principle holding differently for various scenarious used in our problem. Note that the source operator in 2D reduces to 

\begin{align}
    S_f (x, t) = \frac{1}{ 2\pi} \int_{\mathbb{D} (x, t)} \frac{f(x')}{\sqrt{t^2 - \Vert x-x' \Vert^2}} dx'.
\end{align}

\subsection{$3+1$ Yang-Mills fields in $\mathbb{R}^{3+1}$}

Recall that the Yang-Mills potential is a Lie-algebra valued $1-$form 
\begin{align}
    A =& A^{(a)} \mbo{e}_{(a)} dx^\mu 
    \intertext{where, $\mbo{e}_{(a)}$ is a real matrix representation of the Lie algebra of a compact Lie group, written succintly as}
    =& A_\mu dx^\mu
\end{align}
The curvature of the Yang-Mills potential is given by

\begin{align}
    F=& F_{\mu \nu} ^{(a)} \mbo{e}_{(a)} dx^\mu \wedge dx^\nu
    \intertext{succintly}
    =& F_{\mu \nu} dx^\mu \wedge dx^\nu.
    \intertext{The curvature can be expressed interms of the Yang-Mills potential as}
    F_{\mu \nu}=& \ptl_ \mu A_\nu - \ptl_\nu A_\mu + [A_\mu, A_\nu].
\end{align}

The Yang-Mills field equations follow from the variational principle 

\begin{align}
    S_{\text{YM}} = - \frac{1}{4} \int \Vert F \Vert^2 \bar{\mu}_{\bar{g}}
\end{align}
and given by, 

\begin{align}
    \grad_\mu F^{\mu \nu} =0
    \intertext{where the derivative $\grad$ is defined as}
    \grad_\nu F_{\a \b} \fdg= \ptl_\nu F_{\a \b} + [A_\nu, F_{\a \b} ].
\end{align}
Analogous to the Riemann curvature tensor, the Yang-Mills curvature $F$ also satisfies the Bianchi identity: 

\begin{align}
    \grad_\delta F_{\mu \nu} + \grad_\mu F_{\nu \delta} + \grad_{\nu} F_{\delta \mu} =0
\end{align}

which results in the wave propagation equation for the curvature $F$, after taking the derivative and using the field equations: 

\begin{align} \label{YM-curv-prop-mink}
\grad^\a \grad_\a F_{\mu \nu} =& 2[F^\a_\mu, F_{\a \nu}]
\intertext{using the long form of the covariant derivative $\grad$}
\grad^\a \grad_\a F_{\mu \nu} =& \ptl^\a \ptl_\a F_{\mu \nu} + [A^\a, [A_\a, F_{\mu \nu}]] + 2 \ptl_\a ( [A^\a, F_{\mu \nu}]) -[\ptl_\a A^\a, F_{\mu \nu}]
\intertext{for the sake of brevity, let us denote this equation as}
\square F_{\mu \nu} =& N_{\mu \nu} (F)
\end{align}
This is the main equation of study for the Yang-Mills fields in the Minkowski space $\mathbb{R}^{3+1}$. It may also be recalled that the energy-momentum tensor of the Yang-Mills fields is 

\begin{align}
    T_{\mu \nu} = F_{\mu \a} F^{\a}_{\quad \nu} - \frac{1}{4} \bar{g}_{\mu \nu} F_{\a \b} F^{\a \b}. 
\end{align}

Now consider the initial value problem of \eqref{YM-curv-prop-mink}. As we already discussed, the representation formula of the solution of the initial value problem of \eqref{YM-curv-prop-mink} can be decomposed into the linear homogeneous part and the nonlinear part with vanishing initial data. 

\begin{align} \label{duhamel-YM-Mink}
    F= F_L + F_N
\end{align}
where $F_L$ is the homogeneous solution with the initial data 
\begin{subequations}
\begin{align}
    F_L (0, x) =& F(0, x) =\fdg F_0 (x) \\
    \ptl_t F_L(0, x) =& F(0, x) =\fdg F_1(x)
\end{align}
\end{subequations}
and $F_N$ satisfies the equation, 
\begin{align}
  \grad^\a \grad_\a (F_N)_{\mu \nu} =& 2[F^\a_\mu, F_{\a \nu}]
\end{align}
with vanishing initial data.

\begin{subequations}
\begin{align}
    F_N (0, x) =& 0 \\
    \ptl_t F_N(0, x) =& 0.
\end{align}
\end{subequations}

Let us now briefly sketch the method of controlling the $\Vert \cdot \Vert_{L^\infty}$ norm of the curvature by Eardley-Moncrief. We would like to emphasize that the actual field equation is sub-critical with respect to the energy-norm of the Yang-Mills fields obtained from the energy momentum tensor. This structure is a bit different from the energy-critical problem under our consideration. Nevertheless, there are few structural aspects of this Yang-Mills problem that are relevant for our problem. 

The main difficulty in the $\Vert F \Vert_{L^\infty} $ estimate of Eardley-Moncrief is  controlling the bulk nonlinear terms that occur in the representation formula \eqref{duhamel-YM-Mink}. It may be noted that the  nonlinear propagation equation contain commutator terms that are quadratic in curvature. The integral equation involves the terms in the past light cone and as well as the initial data. These integral terms are then estimated in terms of the initial data, in particular the energy, which is conserved in time, and the $\Vert \cdot \Vert_{L^\infty}$ norm of the Yang-Mills curvature itself. Consequently, this estimate provides an integral estimate for the 
$\Vert \cdot \Vert_{L^\infty}$ norm of the curvature, involving the initial data. 

In performing these estimates, the special structure offered by the Cronstr\"om gauge is particularly useful. It was shown separately that the Yang-Mills fields in any gauge can be transformed into the Cronstr\"om gauge. Subsequently,  using energy norms equivalent to  $(H^2 \times H^1),$ it is shown that they cannot blow up in finite time. In this step, the aforementioned $\Vert \cdot \Vert_{L^\infty}$ estimate of the curvature plays an important role, in controlling the energy norms.

Let us discuss how they use the structure of the Yang-Mills field equations to deal with these bulk-integral nonlinear terms,
\begin{align}
    \int  [A^\a, [A_\a, F_{\mu \nu}]] + 2 \ptl_\a ( [A^\a, F_{\mu \nu}]) -[\ptl_\a A^\a, F_{\mu \nu}] + 2[F^\a_\mu, F_{\a \nu}]
\end{align}
term by term
\begin{itemize}
    \item The spacetime bulk-integral quantity: $\displaystyle \int \ptl_\a [A^\a, F_{\a \b}] $ is transformed into a boundary term. Following the use of the gauge condition $x^\mu A_\mu =0$, a part of the boundary term cancels a term in the representation formula of $F_L$, and the remaining boundary terms can be estimated in terms of the initial data in the temporal gauge. 
    
    \item The bulk quantity $\displaystyle \int -[\ptl_\a A^\a, F_{\mu \nu}]  + [A^\a, [A_\a, F_{\mu \nu}]]$ can be transformed using $\displaystyle A_\nu = \int^1_0 x^\mu F_{\mu \nu} (\lambda x) \lambda d \lambda $ in the Cronstr\"om gauge into a quantity involving a cubic term in $F$. This term can in turn be estimated in term of $\Vert F\Vert_{L^\infty}$ using conservation of energy and the fact that $x^\mu F_{\mu \nu} (\lambda x) = r \ell^\mu F_{\mu \nu} (\lambda x) $ in the past null cone. 
    \item Finally, the integrands in the bulk term $\displaystyle \int 2 [F^\a_\mu, F_{\a \nu}]$ can be represented as square integrable terms upon projection into an orthonormal basis. These terms can in turn be estimated in terms of the initial energy and $\Vert F \Vert_{L^\infty}$ norm of the curvature tensor $F.$  
\end{itemize}

Therefore, it follows that the estimate on the representation formula can be schematically represented as 

\begin{align}
    \Vert F \Vert^2_{L^\infty} \leq c(t) \int^t_0 \Vert F \Vert_{L^\infty} + c'(t)
\end{align}
where $c$ and $c'$ are positive functions of depending only on the initial energy, temporal gauge initial data and $t$. Subsequently, establishing that $\Vert F \Vert_{L^\infty}$ is continuous, the Gronwall lemma is applied to show that the $\Vert F\Vert_{L^\infty}$ norm of the Yang-Mills curvature does not blow up in finite time. 

Further, global existence is established using the energy norms 

\begin{align}
    \mathcal{E}_0 \fdg = \halb \int_{\mathbb{R}^3} \Sigma_{i,j} \left( E_i E_i + \ptl_j A_i \ptl_j A_i + m^2 A^2_i  \right)
    \intertext{where $m>0$ is an arbitrary positive constant and}
    \mathcal{E}_1 \fdg = \halb \int_{\mathbb{R}^3} \Sigma_{i,j} 
    \left( \ptl_j E_i \ptl_j E_i + (\ptl_j \ptl_k A_i)^2 \right) 
\end{align}

and the Gronwall type estimates 

\begin{align}
\vert \frac{\ptl \mathcal{E}_0}{\ptl t} \vert  \leq c  \Vert F\Vert_{L^\infty} \mathcal{E}_0
\intertext{and}
\vert \frac{\ptl \mathcal{E}_1}{\ptl t} \vert  \leq c  \Vert F\Vert_{L^\infty} \mathcal{E}_0 + c' \Vert F\Vert_{L^\infty} \mathcal{E}_1.
\end{align}
In the work of Eardley-Moncrief \cite{Eardley_Moncrief_II}, the it is sufficient to show that the energies $\mathcal{E}_0$ and $\mathcal{E}_1$ do not blow up, in view of the fact that the energy norm $(\mathcal{E}_1 + \mathcal{E}_0)^{1/2}$  is equivalent to the $ (H^2 \times H^1 )$  norm of the Yang-Mills solution. 

As we already alluded to, there is close analogy between the Yang-Mills field equations and the Einstein equations via the Cartan formalism. However, from a PDE perspective there are significant differences to be noted. Firstly, the Yang-Mills problem is energy sub-critical and the general $3+1$ Einstein equations are super-critical in nature. Indeed, as we already mentioned an $\Vert \Vert_{L^\infty}$ estimate for the gravitational $F$ tensor is unreasonable, in view of well-known examples where this is violated. A relevant question for the general Einstein equations is the criterion for breakdown of the Einstein equations. There are several important works in this direction \cite{M_anderson_01, KRod_10, QWang_10, KL_RO_SZ_2012} (see especially \cite{KL_RO_SZ_2012} in which it is established that as long as the $L^2$  curvature of the inital data is bounded the solutions for the Einstein equations can be smoothly extend in time).

However, let us recall that, in the case of the $U(1)$ symmetric Einstein equations (i.e., the Einstein equations on a translational $3+1$ dimensional $U(1)$ symmetric spacetime) such an estimate over the entire maximal development of the regular initial data is quite reasonable indeed. In this endeavour, as an intermediary step,  we would like to derive a `break-down' criterion for this step in this case. To prevent any confusion for the reader, we would like to clarify that such a `breakdown' criterion for $3+1$ dimensional $U(1)$ symmetric spacetimes would merely be an intermediary mathematical and technical condition that hopefully can be weakened to be shown eventually that the breakdown of the solutions of this system does not occur at all. 

As a side note, for the Einstein equations on dynamical axisymmetric spacetimes, such as ones `close to' the Kerr black hole spacetimes, we would expect propagation of regularity of initial data in the domain of outer communications, as opposed to the entire maximal development of the initial data. In this case, some of the framework developed in this work would also be useful, but the axisymmetric problem is expected to be considerably more delicate than the current translational case due to regularity and criticality issues. A discussion in this direction for the axisymmetric problem can be found in \cite{NG_21_1}. 

It should also be pointed out that if the spatial Cauchy hypersurface of the $U(1)$ symmetric spacetime $(\bar{M}, \bar{g})$ is compact, then there is a possibility that a big-bang type singularity can form in the past direction. In such cases,  by `global existence'  we mean future global existence (see e.g., \cite{Ycb-Mon_01}).

\section{The Cartan Formalism and Curvature Propagation of Einstein spacetimes}
Now let us turn our attention to the Einstein equations. Consider, the vacuum Einstein equations for now: 

\begin{align}
    \bar{E}_{ \mu \nu} \fdg= \olin{\text{Ric}} _{\mu \nu} - \halb \bar{R} \bar{g}_{\mu \nu}=& 0, \quad (\bar{M}, \bar{g})
    \intertext{equivalently}
    \olin{\text{Ric}}_{\mu \nu} =& 0, \quad (\bar{M}, \bar{g})
\end{align}
It is well-known that the Einstein tensor $\bar{E}_{\mu \nu}$ is divergence-free by virtue of the Bianchi identities on the Lorentzian manifold $(\bar{M}, \bar{g})$. Now then, let us recall the Bianchi identities:

\begin{align}
\intertext{First Bianchi Identity}
    \brie^\a _{\quad \b \gamma \delta} + \brie^\a _{\quad \gamma \delta \b} + \brie^\a_{\quad \delta \b \gamma} =&0, \quad (\bar{M}, \bar{g}) \quad \textbf{(B1)} \\
    \intertext{Second Bianchi Identity}
    \bgrad_\mu \brie^{\a}_{\quad \b \gamma \delta} + \bgrad_\gamma  \brie^{\a}_{\quad \b \delta \mu} + \bgrad_\delta \brie^{\a}_{\quad \b \mu \gamma} =&0, \quad (\bar{M}, \bar{g}) \quad \textbf{(B2)}.
\end{align}

In the spacetime satisfying the Einstein equations, it follows that 

\begin{align}
    \bgrad^\a \brie_{\a \b \gamma \delta} =0, \quad (\bar{M}, \bar{g}). 
\end{align}
This equation can be compared to a similar equation satisfied by the Yang-Mills curvature. Taking a further derivative of the Bianchi identity \textbf{(B1)} and using the properties of the Riemann curvature tensor of $(\bar{M}, \bar{g}),$ we get the 
curvature propagation equation: 

\begin{align}
  \square_{\bar{g}} \, \brie^{\a}_{\quad \b \gamma \delta} \fdg =&  \bgrad^\nu \bgrad_\nu \brie^{\a}_{\quad \b \gamma \delta} \notag\\
  =& 2 \brie^\a_{\quad \mu \delta \sigma}  \brie^{\,\,\mu \,\, \sigma}_{\b \,\, \gamma} -2 \brie^{\sigma}_{\quad \mu \gamma \delta}  \brie^{\,\,\mu \,\, \sigma}_{\b \,\, \delta} \notag\\
    & - \brie^{\quad \mu \sigma}_{ \gamma \delta} \brie^\a_{\quad \b \mu \sigma}, \quad (\bar{M}, \bar{g}).
\end{align}
This equation, referred to as the `Penrose wave equation', is the main equation of study in our proposed approach. Let us now rewrite this equation using the Cartan formalism. 

In the Cartan formalism, the metric is not the main object for the representation of the geometric structures on these spacetimes $(\bar{M}, \bar{g}).$ The geometric objects are constructed from a tetrad.  We shall represent the basis vectors as $\{ \mbo{e}_{\hatt{a}}, \, \hatt{a}=1, 2, 3, 4\}$ and the dual basis vectors as $ \{ \mbo{e}^{\hatt{b}},\, \hatt{b}=1,2,3,4 \} $ so that 
\begin{align}
\ip{\mbo{e}^{\hatt{b}}}{\mbo{e}_{\hatt{a}}} = \delta^{\hatt{b}}_{\hatt{a}}
\end{align}
Let us define the torsion, $T$ on the principal bundle manifold $\bar{M}$  (using the Levi-Civita connection)as 
\begin{align}
T (X,Y) = \leftexp{(\bar{g})}{\grad}_X Y - \leftexp{(\bar{g})}{\grad}_Y X - [X, Y]
\end{align}
where $X, Y \in T(M)$ and $[\cdot, \cdot]$ is the Lie bracket in $\bar{M}.$ In a coordinate basis, this can be represented in general as 
\begin{align}
T^{\hatt{a}}_{\hatt{b} \hatt{c}} = \leftexp{(\bar{g})}{\Gamma}^{\hatt{a}}_{\hatt{b} \hatt{c}} - \leftexp{(\bar{g})}{\Gamma^{\hatt{a}}_{\hatt{c}\hatt{b}}} - C^{\hatt{a}}_{\hatt{b} \hatt{c}}
\end{align}
where $C^{\hatt{a}}_{\hatt{b} \hatt{c}}$ are the structure constants: 
\begin{align}
C^{\hatt{a}}_{\hatt{b} \hatt{c}} = \ip{\mbo{e}^{\hatt{a}} }{[\mbo{e}_{\hatt{a}}, \mbo{e}_{\hatt{c}}]}
\end{align}
of the basis $\mbo{e}_{\hatt{a}}.$
Now the connection $1-$forms or the spin connection in Cartan formalism is defined as 
\begin{align}
\Theta^{\hatt{a}}_{\hatt{c}} \fdg= \leftexp{(g)}{\Gamma}^{\hatt{a}}_{\hatt{b} \hatt{c}} \mbo{e}^{\hatt{a}}
\end{align} 
and it is well known that 
\begin{align}
d g_{ \hatt{a} \hatt{b}} = \Theta_{\hatt{b} \hatt{a}} + \Theta_{\hatt{a} \hatt{b}} 
\end{align}
It may be recalled in the our chosen orthonormal basis, we have 
\begin{align} \label{spin-sum-identity}
 \Theta_{\hatt{b} \hatt{a}} + \Theta_{\hatt{a} \hatt{b}}  =0
\end{align}
Now then, we shall use the Cartan's first structural equation, that relates the torsion with the spin-connections 

\begin{align}
T^{\hatt{a}}  = d \omega^{\hatt{a}} + \Theta^{\hatt{a}}_{\hatt{c}} \wedge \mbo{e}^{\hatt{c}}.
\end{align}
In case the torsion vanishes, we have
\begin{align}
d \omega^{\hatt{a}} = - \Theta^{\hatt{a}}_{\hatt{c}} \wedge \mbo{e}^{\hatt{c}}.
\end{align}
We shall use the above equation, together with \eqref{spin-sum-identity} throughout this work. 
Now then, let us define the curvature $ \brie \fdg T_p(\bar{M}) \times T_p (\bar{M}) \times T_p(\bar{M}) \to T_p (\bar{M})$ as 
\begin{align}
\brie (X,Y) Z =  \left( \leftexp{(\bar{g})}{\grad}_X   \leftexp{(\bar{g})}{\grad}_Y -  \leftexp{(\bar{g})}{\grad}_Y  \leftexp{(\bar{g})}{\grad}_X  \right)  Z  
\end{align}
which can be represented as a `curvature' $2-$ in the Cartan formalism, in terms of the dual basis vectors as 
\begin{align}
\brie^{\, \hatt{a}}_{\,\,\,\, \hatt{b}}  = - \halb \text{Riem}^{\,\hatt{a}}_{\,\,\,\, \hatt{b} \hatt{c} \hatt{d}} \, \mbo{e}^{\hatt{c}} \mbo{e}^{\hatt{d}}
\end{align}
Now, then in terms of the spin connection, we have
\begin{align}
\brie^{\,\hatt{a}}_{\,\,\,\,\hatt{b}}  = d\, \Theta ^{\hatt{a}}_{\,\,\,\, \hatt{b}}  + \Theta ^{\hatt{a}}_{\,\,\,\, \hatt{c}} \wedge \Theta ^{\hatt{a}}_{\,\,\,\, \hatt{b}}
\end{align}
which is the so-called Cartan's second structural equation.

Now consider, in Cartan formalism, 

\begin{align}
    \bgrad_\a \brie^{\hatt{a}}_{\quad \hatt{b} \mu \nu} =& \mbo{e}^{\hatt{a}}_\lambda \mbo{e}^\sigma_{\hatt{b}}  \bgrad_\a \brie^\lambda_{\quad \sigma \mu \nu} \\
    =& \bgrad_\a \brie^{\hatt{a}}_{\quad \hatt{b} \mu \nu} + \Theta^{\hatt{a}}_{\quad \hatt{c} \a} \brie^{\hatt{c}}_{\hatt{b} \mu \nu} - \Theta^{\hatt{c}}_{\quad \hatt{b} \a} \brie^{\hatt{a}}_{\quad \hatt{c} \mu \nu} .
\end{align}

It may be noted that we can treat $\displaystyle \brie^{\hatt{a}}_{\quad \hatt{b} \mu \nu} $ as the analog of the Yang-Mills curvature $F_{\mu \nu}$ and likewise $\Theta^{\hatt{a}}_{\quad \hatt{b} \mu}$ as the  analog of the Yang-Mills vector potential $A_\mu.$ Thus, we denote them as such for the sake of brevity and convenience. In Cartan formalism, the gravitational curvature wave equation can be represented as 

\begin{align}
    \bgrad^\a F_{\a \b} =0, \quad (\bar{M}, \bar{g})
\end{align}
where $(\bar{M}, \bar{g})$ satisfies the Einstein equations and the relation between the gravitational curvature $F_{\mu \nu}$ and the gravitational `vector potential' $A_\mu$ is 

\begin{align}
F^{\hatt{a}}_{\,\, \hatt{b} \mu \nu} = \ptl_\mu A^{\hatt{a}}_{\,\, \hatt{b} \mu} - \ptl_\nu A^{\hatt{a}}_{\,\, \hatt{b} \nu} + A^{\hatt{a}}_{\,\, \hatt{c} \mu} A^{\hatt{c}}_{ \,\, \hatt{b} \nu}  - A^{\hatt{a}}_{\,\, \hatt{c} \nu} A^{\hatt{c}}_{ \,\, \hatt{b} \mu} 
\end{align}
which again is analogous in structure to the Yang-Mills curvature and vector potentials. This structure has to be compared to Cartan's first structural equation. 

This form of representation shall be helpful for us later. In particular, we shall observe that in the Kaluza-Klein ansatz, we shall note that a certain combination of the gravitational F-tensor is independent of the undifferentiated `vector potential' tensor. Let us now consider the wave propagation equation for the gravitational tensor $F$:

\begin{align}
    \square_{\bar{g}} F_{\mu \nu} + \brie^{\quad \a \b }_{\mu \nu} F_{\a \b} = N_{\mu \nu} (\Theta, \brie), \quad (\bar{M}, \bar{g})
\end{align}

where $N_{\mu \nu} (\Theta, \brie)$ on the right hand side is a non-linearity. 

expanding out, using Cartan formalism, this can be rewritten as 

\begin{align}
    &\bar{g}^{\a \b} \Big( \bgrad_\b \big( \bgrad_{\hatt{a}}  \brie^{\hatt{a}}_{\quad \hatt{b} \mu \nu} + \Theta^{\hatt{a}}_{\,\,\hatt{c} \a} \brie^{\hatt{c}}_{\quad \hatt{b} \mu \nu} - \Theta^{\hatt{a}}_{\,\, \hatt{b} \a} \brie^{\hatt{a}}_{\quad \hatt{c} \mu \nu} \big) 
    \notag\\
    &+ 
    \Theta^{\hatt{a}}_{\,\, \hatt{c} \b} (\bgrad_\a \brie^{\hatt{c}}_{\quad \hatt{b} \mu \nu}  
    + \Theta^{\hatt{c}}_{\,\, \hatt{d} \a}  \brie^{\hatt{d}}_{\quad \hatt{b} \mu \nu} 
    -\Theta^{\hatt{d}}_{\,\, \hatt{b} \a}  \brie^{\hatt{c}}_{\quad \hatt{d} \mu \nu} ) \notag\\
    &- \Theta^{\hatt{c}}_{\,\, \hatt{b} \b} (\bgrad_\a \brie^{\hatt{a}}_{\quad \hatt{c} \mu \nu} 
    + \Theta^{\hatt{a}}_{\,\, \hatt{d} \a}  \brie^{\hatt{d}}_{\quad \hatt{c} \mu \nu}
    -\Theta^{\hatt{d}}_{\,\, \hatt{c} \a}  \brie^{\hatt{a}}_{\quad \hatt{d} \mu \nu}
    )    \Big)  \notag\\
    & = - \brie_{\mu \nu}^{\quad \gamma \delta} \brie^{\hatt{a}}_{\quad \hatt{b} \gamma \delta} + 
    2 \brie^{\hatt{a}}_{\quad \hatt{c} \mu \delta} \brie^{\hatt{c}\quad\delta}_{\,\,\,\hatt{b} \nu} 
    \notag\\
    & \quad -2 \brie^{\hatt{a}}_{\quad \hatt{c} \nu \delta} \brie^{\hatt{c}\quad\delta}_{\,\,\,\hatt{b} \mu} 
\end{align}
 
 rearranging the terms we can rewrite this equation in the form 
\begin{align}
    \bar{g}^{\a \b} \bgrad_\a \bgrad_\b \brie^{\hatt{a}}_{\quad \hatt{b} \mu \nu} + \brie_{\mu \nu}^{\quad \gamma \delta} \brie^{\hatt{a}}_{\quad \hatt{b} \gamma \delta}  = N_{\mu \nu} (\Theta, \brie)
\end{align}

where $N_{\mu \nu}$ is 

\begin{align}
   N_{\mu \nu} =  &-\bar{g}^{\a \b} \Big( \bgrad_\b \big( \Theta^{\hatt{a}}_{\,\,\hatt{c} \a} \brie^{\hatt{c}}_{\quad \hatt{b} \mu \nu} - \Theta^{\hatt{a}}_{\,\, \hatt{b} \a} \brie^{\hatt{a}}_{\quad \hatt{c} \mu \nu} \big) 
    \notag\\
    &+ 
    \Theta^{\hatt{a}}_{\,\, \hatt{c} \b} (\bgrad_\a \brie^{\hatt{c}}_{\quad \hatt{b} \mu \nu}  
    + \Theta^{\hatt{c}}_{\,\, \hatt{d} \a}  \brie^{\hatt{d}}_{\quad \hatt{b} \mu \nu} 
    -\Theta^{\hatt{d}}_{\,\, \hatt{b} \a}  \brie^{\hatt{c}}_{\quad \hatt{d} \mu \nu} ) \notag\\
    &- \Theta^{\hatt{c}}_{\,\, \hatt{b} \b} (\bgrad_\a \brie^{\hatt{a}}_{\quad \hatt{c} \mu \nu} 
    + \Theta^{\hatt{a}}_{\,\, \hatt{d} \a}  \brie^{\hatt{d}}_{\quad \hatt{c} \mu \nu}
    -\Theta^{\hatt{d}}_{\,\, \hatt{c} \a}  \brie^{\hatt{a}}_{\quad \hatt{d} \mu \nu}
    )    \Big)  \notag\\
    &  + 
    2 \brie^{\hatt{a}}_{\quad \hatt{c} \mu \delta} \brie^{\hatt{c}\quad\delta}_{\,\,\,\hatt{b} \nu} 
     -2 \brie^{\hatt{a}}_{\quad \hatt{c} \nu \delta} \brie^{\hatt{c}\quad\delta}_{\,\,\,\hatt{b} \mu}.
\end{align}

Subsequently, using the notation $\displaystyle F_{\mu \nu} = \brie^{\hatt{a}}_{\quad \hatt{b} \mu \nu}$, drawing on the Yang-Mills analogy, we can recover the analog of the Yang-Mills field equations, 

\begin{align}
  \bgrad^\a \bgrad_\a  F_{\mu \nu} + \brie_{\mu \nu}^{\quad \gamma \delta} F_{\gamma \delta} = N_{\mu \nu} (\Theta, \brie).
\end{align}

\subsection{An Integral Equation for the Curvature}

Analogous to the wave equation on the Minkowski space already discussed earlier in our article, we can construct an integral formula for the curvature wave equation. The integral equation is given by: 

\begin{align}
F_{\mu \nu} (x)=& \frac{1}{2 \pi} \int_{D_p}  (V^+)^{ \mu' \nu'}_{\mu \nu} (x, x') N_{\mu' \nu'} (x') \bar{\mu} (x') + \int_{C_p} U^{\mu' \nu'}_{\mu \nu} (x, x') N_{\mu' \nu'} (x') \bar{\mu}_{\Gamma} (x')  \notag\\
& + \frac{1}{2 \pi} \int_{S_p} \star  \big ( (V^+)^{ \mu' \nu'}_{\mu \nu} (x, x') \grad^{\gamma'} F_{\mu' \nu'} (x') - F_{\mu' \nu'} \grad^{\gamma'} (V^+)^{ \mu' \nu'}_{\mu \nu} (x, x') \big)   \notag \\
& + \frac{1}{2 \pi} \int_{\sigma_p} U^{\mu' \nu'} (x, x') _{\mu \nu} (2 \grad^\gamma t(x') F_{\mu' \nu'} (x') + F_{\mu' \nu'} (x') \square t (x') ) \notag\\
&- \ip{\grad t}{\grad' \Gamma} (V^+)^{\mu \nu}_{\mu \nu} (x, x') (x') ) \bar{\mu}_{t, \Gamma} (x')
\end{align}
$\bar{\mu}_{\Gamma}$ is a Leray form, $U^{\mu' \nu'}_{\mu \nu} (x, x') = \kappa (x, x') \tau^{\mu' \nu'}_{\mu \nu},$ $\Gamma$ is an optical function that satisfies the Eikonal equation. 

In view of the self-adjoint nature of the linear operator $L$ and the reciprocity between the forward fundamental solution and the backward fundamental solution (Theorem 5.2.2 in \cite{Friedlander_1975}), the tensorial quantity $V^+ (x, x')$ is the solution of a characteristic initial value problem for the homogeneous wave equation.  

\begin{align}
    \square_{\bar{g} (x')} (V^+)^{\mu \nu}_{\a \b} + \rie^{\mu' \nu'}_{\quad \delta' \gamma'} (x') (V^+)^{\gamma' \delta'} + \text{Ric}(x') V^+ (x, x')-\text{Ric} (x') V^+ (x, x') =0
\end{align}
the initial data on the null cone $C_p$ is denoted as $V_0$ which satisfies a transport equation \cite{Moncrief_2005}. 

Now let us look at the following transformations of Moncrief \cite{Moncrief_2005} to deal with the bulk terms. The a priori transformations of the bulk terms are based on the integral transformations and Stokes' theorem. 

Consider the identity, 

\begin{align}
    &\bgrad_{\gamma'}\left(  (V^+)^{\mu \nu}_{\a \b} \bgrad^{\gamma'}
  F_{\mu' \nu'}  \right) = \left( \bgrad_{\gamma'} (V^+)_{\a \b}^{\mu' \nu'} \right) \bgrad^{\gamma'} F_{\mu' \nu'} + 
  (V^+)^{\mu' \nu'}_{\a \b} \bgrad_{\gamma'} \bgrad^{\gamma'} F_{\mu' \nu'} \notag\\
  \intertext{further}
  &= \bgrad_{\gamma'} \bgrad^{\gamma'} (V^+)^{\mu' \nu'}_{\a \b} F_{\mu' \nu'} + \bgrad^{\gamma'} (V^+)^{\mu' \nu'}_{\a \b} \bgrad F_{\mu' \nu'} + (V^+)^{\mu' \nu'}_{\a \b} \bgrad_{\gamma'} \bgrad^{\gamma'} F_{\mu' \nu'}
  \end{align}

 Recall the (tensorial) linear wave operator, $\displaystyle Lu_{\a \b} \fdg =  \bgrad_\gamma \bgrad^\gamma \, u_{\a \b} + \brie_{\a \b}^{\quad \gamma \delta} u_{\gamma \delta}. $

Consider then,

\begin{align}
    \bgrad_{\gamma'} ( (V^+)^{\mu' \nu'}_{\a \b} \bgrad^{\gamma'} F_{\mu' \nu'} ) = \bgrad_{\gamma'} (V^+)^{\mu' \nu'}_{\a \b} \bgrad^{\gamma'} F_{\mu' \nu'} + (V^+)^{\mu' \nu'}_{\a \b} \bgrad_{\gamma'} \bgrad^{\gamma'} F_{\mu' \nu'}.
\end{align}
Likewise, consider, 

\begin{align}
     \bgrad_{\gamma'} \bgrad^{\gamma'} (V^+)^{\mu \nu'} F_{\mu' \nu'} + \bgrad^{\gamma'} (V^+)^{\mu' \nu'}_{\a \b} \bgrad_\gamma F_{\mu \nu'}.
\end{align}

Thus, we have,

\begin{align}
    (V^+)^{\mu' \nu'}_{\a \b} L F_{\mu' \nu'} = 
    \big( (V^+)^{\mu' \nu'}_{\a \b} \square_{\bar{g}} F_{\mu' \nu'} + \brie^{\mu' \nu'}_{\quad \gamma \delta} F_{\gamma \delta} (V^+)_{\a \b} \big).
\end{align}
It follows that, 

\begin{align}\label{F-div-term}
    (V^+)^{\mu' \nu'}_{\a \b} N_{\mu' \nu'} =& (V^+)^{\mu' \nu'}_{\a \b} L F_{\mu' \nu'} - (L V^+) ^{\mu' \nu'}_{\a \b} F_{\mu' \nu'} (x') \notag\\
    =& (V^+)^{\mu' \nu'}_{\a \b} \square_{\bar{g}} F_{\mu' \nu'} -\square_{\bar{g}} (V^+)^{\mu' \nu'}_{\a \b} F_{\mu' \nu'} \notag\\
    =& \bgrad_{\gamma'} (V^{\mu' \nu'}_{\a \b} \bgrad^{\gamma'} F_{\mu' \nu'} - \bgrad^{\gamma'} (V^+)^{\mu' \nu'}_{\a \b} F_{\mu' \nu'} ).
\end{align}

Crucially, it may be noted that the final term in \eqref{F-div-term} is a purely spacetime divergence term. Also note that these transformations do not need the fact $V^+$ is a solution of the characteristic initial value problem with the data given by $V_0.$

Let us now decompose these `tail' contributions are follows 

\begin{align}
    F^{\text{tail}}_{\a \b} = \leftexp{I}{F^{\text{tail}}_{\a \b}} + \leftexp{II}{F^{\text{tail}}_{\a \b}}
\end{align}
where we have followed the terminology of Friedlander-Moncrief. We have, 

\begin{align}
    \leftexp{I}{F^{\text{tail}}_{\a \b}} = \frac{1}{ 2 \pi} \int_{C_p}  \bgrad^{\gamma} \Gamma(x, x') \Big( (V^+)^{\mu' \nu'}_{\a \b} (x, x') \bgrad_{\gamma'} F_{\mu' \nu'} (x') - F_{\mu' \nu'} (x') \bgrad_{\gamma'} (V^+)^{\mu' \nu'}_{\a \b} (x, x') \Big) \gamma_{\Gamma} (x')
\end{align}

which can be transformed to the following: 

\begin{align}
   & \leftexp{I}{F^{\text{tail}}_{\a \b}}(x) \notag\\
   &= \frac{1}{2 \pi} \int_{C_p} \Big( \bgrad^{\gamma'} \Gamma \bgrad_{\gamma'} (V_0)^{\mu' \nu'}_{\a \b} (x, x') F_{\mu' \nu'} (x') + F_{\mu' \nu'} (L U^{\mu' \nu'}_{\a \b} (x, x') \\
   &+ (\square_{\bar{g}} \Gamma -4) (V_0)^{\mu' \nu'}_{\a \b} (x, x') ) \Big) \bar{\mu}_{\Gamma} (x').
\end{align}

Defining the following null coordinates, using the normal coordinates $\{ (x^\mu)' \}$  as follows: 
\begin{align}
    \{ (x^\mu)' \}  \longleftrightarrow \{x^{0'}, r', \theta, \phi \}
\end{align}
as 
\begin{subequations}
\begin{align}
    \xi = t'+ r', \quad \eta= t'-r' \\
    t'= \frac{\xi+ \eta}{2}, \quad r'= \frac{\xi-\eta}{2}.
\end{align}
\end{subequations}
We have the following decomposition
\begin{align}
    \bar{\mu}_{\Gamma} =& \frac{\sqrt{\vert -\bar{g}\vert}}{u} du \wedge d\theta \wedge d \phi  \notag\\
    \bar{\mu}_{\bar{g}} =& d \Gamma \wedge \bar{\mu}_{\Gamma} = \sqrt{- \vert \bar{g} \vert} d u \wedge d v \wedge d \theta \wedge d \phi .
\end{align}
Now the first term $\leftexp{I}{F^{\text{tail}}_{\a \b}}$ is 

\begin{align}
    &\frac{1}{ 2 \pi} \int_{C_p} \Big( \bgrad^{\gamma'} \Gamma \bgrad_{\gamma'} ( (V_0)^{\mu' \nu'} F_{\mu' \nu'} (x')) \Big) \bar{\mu}_\Gamma (x') \notag\\
   & =  \frac{1}{ 2 \pi} \int_{C_p} \Big( (4-\square_{\bar{g}} \Gamma) (V_0)^{\mu' \nu'}_{\a \b} F_{\mu' \nu'}(x')\Big) \bar{\mu}_{\Gamma} (x') \notag \\
   &\quad - \frac{1}{2 \pi} \int_{\sigma_p} \Big(2 \sqrt{\bar{g}_{\gamma \delta}} (V_0)^{\mu' \nu'}_{\a \b} (x, x') F_{\mu' \nu'} (x') \Big) d \theta \wedge d \phi  \\
   &=  \frac{1}{ 2 \pi} \int_{C_p} \Big( (4-\square_{\bar{g}} \Gamma) (V_0)^{\mu' \nu'}_{\a \b} F_{\mu' \nu'}(x')\Big) \bar{\mu}_{\Gamma} (x') \notag \\
   &- \frac{1}{2 \pi} \int_{\sigma_p} 
   \Big( \sqrt{g_{AB}} (V_0)^{\mu' \nu'}_{\a \b} (x, x') F_{\mu' \nu'} (x') \Big) d \theta \wedge d \phi .
\end{align}
As a consequence, we have the following expression for $\leftexp{I}{F^{\text{tail}}_{\a \b}}$: 

\begin{align}
    \leftexp{I}{F^{\text{tail}}_{\a \b}}=  \frac{1}{ 2 \pi} \int_{C_p} \Big( LU^{\mu' \nu'}_{\a \b} (x, x') F_{\a \b} (x')  \Big) \bar{\mu}_{\Gamma} (x')  - \frac{1}{ 2 \pi} \int_{\sigma_p} ( (V_0)^{\mu' \nu'} (x, x') F_{\mu' \nu'} (x)) \bar{\mu}_{\sigma_p}
\end{align}
The final expression for the tail terms, after eliminating the bulk-terms is: 

\begin{align}
    F^{\text{tail}}_{\a \b}= \frac{1}{2 \pi} \int_{C_p} \Big( LU^{\mu' \nu'}_{\a \b} (x, x') F_{\mu' \nu'} (x) \Big) \bar{\mu}_\Gamma (x').
\end{align}

Now then, the integral formula of the (gravitational) curvature tensor is 

\begin{align}
    F_{\a \b } (x) =& \frac{1}{ 2 \pi} \int_{C_p} \Big( LU^{\mu' \nu'}_{\a \b} (x, x')  F_{\mu' \nu'} (x) + U^{\mu' \nu'}_{\a \b} (x, x') N_{\mu' \nu'} (x') \Big) \bar{\mu}_{\Gamma} (x') \notag\\
    & + \frac{1}{ 2 \pi} \int_{\sigma_p} U^{\mu'\nu'}_{\a \b} \Big( 2 \bgrad t(x') \bgrad_{\gamma'} F_{\mu' \nu'} (x') + F_{\mu' \nu'} (x') \square_{\bar{g}'} t(x') \Big) \bar{\mu}_{\sigma_p} (x').
\end{align}
Let us summarize these results:
\begin{theorem}
There exists an integral formula for a solution of the Penrose wave equation, 

\begin{align}
    \square_{\bar{g}} F_{\mu \nu} + \brie^{\quad \a \b}_{\mu \nu}
    = N_{\mu \nu} (\Theta, \brie), \quad (\bar{M}, \bar{g}),
\end{align}

$\mu, \nu, \a, \b = 0, 1, 2, 3,$ such that the integral formula depends only on the initial data, the mantel of the past null cone and the intersection of the mantel of the past null cone and initial data. \end{theorem}
\subsection{Weyl Fields and the Bel-Robinson Tensor}
As we already alluded to, the standard stress energy tensor for the Yang-Mills equations identically vanishes for the gravitational case. Therefore, we shall use the Bel-Robinson tensor instead. The Bel-Robinson tensor retains or even surpasses  several desirable properties of the stress energy tensor such as symmetry and vanishing spacetime divergence. 

The Weyl curvature tensor and the Bel-Robinson tensor are given by
\begin{align}
    \olin{W}_{\a \b \gamma \delta} =& \brie_{\a \b \gamma \gamma} + \halb (\bar{g}_{\a \delta} \bric_{\gamma \b} + \bar{g}_{\b \gamma} \bric_{\delta \a} - \bar{g}_{\a \delta} \bric_{\delta \b} - \bar{g}_{\b \delta} \bric_{\gamma \a} ) \notag\\
    & + \frac{1}{6} (\bar{g}_{\a \gamma} \bar{g}_{\delta \b} - \bar{g}_{\a \delta} \bar{g}_{\gamma \b}) \bar{R}
\end{align}
and 
\begin{align}
    \olin{Q}_{\a \b \gamma \delta} =& \brie_{\a \mu \gamma \nu} \brie^{\,\,\mu \,\, \nu}_{\b \,\, \delta} 
    + \brie_{\a \mu \delta \nu} \brie^{\,\,\mu \,\, \nu}_{\b \,\,\gamma} \notag\\
    &- \frac{1}{8} \bar{g}_{\a \b} \bar{g}_{\gamma \delta} \brie_{\mu \nu \sigma \zeta} \brie^{\mu \nu \sigma \zeta}.
\end{align}
which can be compactly represented in terms of the Weyl tensor in the form: 

\begin{align}
 \olin{Q}_{\a \b \gamma \delta} = W_{\a \sigma \gamma \zeta} W^{\,\,\sigma\,\, \zeta}_{\a \,\,\delta} 
    - \frac{3}{2} \bar{g}_{a [\b} W_{\mu \nu] \gamma \zeta} W^{\mu \nu \,\,\,\, \zeta}_{\quad \delta}.
\end{align}

The advantage of the Bel-Robinsor tensor  $\olin{Q}_{\a \b \gamma \delta}$ is that it satisfies the following properties
\begin{itemize}
\item It is a totally symmetric tensor
    \item The contraction $Q(X, X, Y, Y)$ is a non-negative quantity for the future directed timelike vector fields $X, Y.$
    \item The Bel-Robinson tensor is spacetime divergence-free i.e., $\displaystyle \bgrad^\nu \olin{Q}_{\a \b \mu \nu} \equiv 0$. This is a consequence of the Bianchi identities for $\brie$ defined in $(\bar{M}, \bar{g}).$ 
    \item Suppose $\displaystyle \leftexp{(X)}{J}^\nu \fdg= Q^{\quad \,\, \nu}_{\a \b \mu} X^\a X^\b X^\mu.$ Then, using the fact that $Q$ is divergence-free, 
    \begin{align}
        \bgrad_\nu \leftexp{(X)}{J}^\nu =  \frac{3}{2} Q_{\a \b \mu \nu} X^\a X^\b (\bgrad^\a X^\b + \bgrad^\b X^\a)
    \end{align}
    Note the occurance of the Killing form  $\mathcal{L}_X \bar{g}$ in the expression above,
    \begin{align}
        \mathcal{L}_X \bar{g} = \bgrad_\a  X_\b + \bgrad_\b X_\a. 
    \end{align}

\end{itemize}
 It may be noted that, in the nonlinear gravitational problem, energies that are conserved or bounded in time are scarce. At a first glance, in order to achieve this control, we need Killing vector fields, conformal Killing vector fields or the like, for the spacetime. However, this is requirement is too strong and will almost certainly exclude interesting cases. 

In the case of translational and equivariant $U(1)$ problem, it was shown that there exists a conserved energy that effectively corresponds to a time-like vector field, which is not a Killing vector field \cite{diss_13, AGS_15}. This fortuitous behaviour was used in a fundamental manner in the aforementioned works. However, this outcome can be attributed to the special structure provided by the equivariance assumption and it is unlikely that this result holds without this assumption. In an attempt to overcome this issue, the  general notion of quasi-local approximate Killing vector fields  was developed in \cite{Moncrief_2005, Moncrief_2014}. The framefields in the Cartan formalism, when subjected to parallel propagation, can be shown to satisfy the Killing equation approximately, with an error term that is explicitly expressible in terms of the curvature and goes to zero as we approach the tip of the cone. We hope that this concept will be helpful in our approach. 

\subsection{Quasi-local Killing and Conformal Killing operators}

Recall the covariant derivative in Cartan formalism, 

\begin{align}
\bgrad_\nu X^{\hatt{a}}_{\mu}=& - \Theta^{\hatt{a}}_{\quad \hatt{b} \nu} X^{\hatt{b}}_{\mu}  \\
\intertext{equivalently, using Christoffel symbols}
=& \ptl_\nu X^{\hatt{a}} - \olin{\Gamma}^{\lambda}_{\mu \nu} X^{\hatt{a}}_\lambda
\end{align}
As a consequence, we have the Killing form, 

\begin{align}
    K_{\mu \nu} (\bar{g}) \fdg =&  \bgrad_\nu X^{\hatt{a}}_{\mu} - \bgrad_\nu X^{\hatt{a}}_{\mu} \notag\\
    =&  - \Theta^{\hatt{a}}_{\,\,\hatt{b} \nu} X^{\hatt{b}}_{\mu} - \Theta^{\hatt{a}}_{\,\,\hatt{b} \mu}  X^{\hatt{b}}_\nu
\end{align}
In the frame field formalism, we shall use the quasi-local (conformal) Killing vector fields. The advantage of this construction is that the spacetime connection can be recovered in terms of the spacetime curvature. 

\begin{lemma} In the Cronstr\"om gauge, we have 
\begin{align}
\Theta^{\hatt{c}}_{\,\, \hatt{a} \mu} = - \int^1_0 \brie^{\hatt{c}}_{\quad \hatt{c} \mu \nu} (\lambda x) \lambda x^\nu d \lambda.
\end{align}
for the frame fields $\{ \mbo{e}_i \}$, we have 

\begin{align}
[ \mbo{e}_{\hatt{a}}, \mbo{e}_{\hatt{b}} ]^\nu =& \mbo{e}^\mu_{\hatt{a}} \bgrad_\mu \mbo{e}^\nu_{\hatt{b}} - \mbo{e}^\mu_{\hatt{b}} \bgrad_\mu \mbo{e}^\nu_{\hatt{a}}  \notag\\
= & \mbo{e}^\nu_{\hatt{c}} (\mbo{e}^\mu_{\hatt{a}} \Theta^{\hatt{c}}_{{\,\, \hatt{b} \mu}} - \mbo{e}^\mu_{{\hatt{b}}} \Theta^{\hatt{c}}_{{\,\, \hatt{a} \mu}} )
\end{align}

\end{lemma}

For example, this construction is applied to the optical function $\bar{f}$. The integral curves of $\bgrad \bar{f}$ are causal inside the null cone i.e., $\displaystyle \bar{g}^{\a \b} \ptl_\a \bar{f} \ptl_\a \bar{f} = 4 \bar{f}$

We have,

\begin{align}
    \bgrad^\b \bar{f} = \ptl^\b \bar{f} = 2 \bar{g}^{\a \b} \bar{g}_{\a \nu}  x^\nu = 2 x^\b 
\end{align}

Likewise, the Killing and conformal Killing forms of $\bar{f}$ can be computed as 

\begin{align}
    K_{\a \b} (\bgrad \bar{f}, \bar{g}) =& (\bgrad_\b \ptl_\a + \bgrad_\a \ptl_b) \bar{f}
    \notag\\
    =& 4 \bar{g}_{\a \b} + 2 x^\nu \ptl_\nu \bar{g}_{\a \b} \\
    CK_{\a \b} (\bgrad \bar{f}, \bar{g}) =& (\bgrad_\b \ptl_\a + \bgrad_\a \ptl_b - \halb \bar{g}_{\a \b} \bar{g}^{\mu \nu} \grad_{\nu} \ptl_\mu) \bar{f}  \notag \\
    =& 2 x^\nu \ptl_\nu \bar{g}_{\a \b} - \halb \bar{g}_{\a \b} \bar{g}^{\gamma \delta} (x^\nu \ptl_\nu \bar{g}_{\gamma \delta})
\end{align}

As a consequence, we have, 
\begin{align}
    \square \bar{f} = 8 + 2 x^\nu \frac{1}{\sqrt{-
\bar{g}}} \ptl_\nu (\sqrt{- \bar{g}}).
\end{align}

It can be established that the commutator of the $\bgrad \bar{f}$ with the framefields has the following structure: 

\begin{align}
    [ \mbo{e}_{\hatt{a}}, \bgrad \bar{f}]^\nu = 2 \mbo{e}^{\nu}_{\hatt{a}} + \mbo{e}^\mu_{\hatt{a}} \bar{g}^{\lambda \nu} x^\b \ptl_\b \bar{g}_{\lambda \nu}
\end{align}
from which it can be seen that the error terms are quadratically vanishing. 

\section{Einstein Equations on Spacetimes with Spacelike U(1) Symmetry}
Consider a $3+1$ dimensional Lorentzian spacetime $(\bar{M}, \bar{g})$ such that it admits a spacelike isometry $X \fdg = \ptl_{x^3}$ such that the metric $\bar{g}$ admits a decomposition of the form: 

\begin{align}
\bar{g} = \tilde{g} + e^{2 \gamma} (dx^2 + A_\mu dx^\mu), \quad \mu = 0, 1, 2. 
\end{align}

If we consider a further conformal transformation for the metric $\tilde{g}$ in the reduced orbit space, 

\begin{align}
g = e^{2 \gamma} \tilde{g}, \quad  \text{in the reduced space} \quad  (M)
\end{align}

In the reduced orbit space manifold, the vacuum Einstein equations for general relativity,  in Lagrangian form 

\begin{align}
\olin{R}_{\mu \nu} = 0, \quad (\bar{M}, \bar{g})
\end{align}

can be reduced to the Einstein-wave map system in $(M, g)$

\begin{subequations}\label{reduced-einstein-wavemap}
\begin{align}
\leftexp{(g)}{\text{Ein}}_{\mu \nu} =& T_{\mu \nu} \label{reduced-einstein}\\
\square_g \gamma  + \halb e^{-4 \gamma} g^{\a \b} \ptl_a \omega \ptl_b \omega  = 0&  \label{wavemap-gamma}\\
\square_g \omega - 4 e^{-4 \gamma} g^{\a \b} \ptl_a \gamma \ptl_b \omega = 0 \label{wavemap-omega}&  
\end{align}
\end{subequations}

where $\square_g $ is the covariant wave operator in the orbit space $(M, g).$ $E$ and $T$ are the Einstein tensor and the energy-momentum tensor of the wave map $$U \fdg (M, g) \to (N, h), \, (N, h)\,\,  \text{is the hyperbolic 2-plane} $$
with  $ h = 4 d\rho^2 + e^{-4 \rho} d\theta^2 $, respectively, 

\begin{align}
\leftexp{(g)}{\text{Ein}}_{\mu \nu} =& \leftexp{(g)}{\text{Ric}}_{\mu \nu} - \halb g_{\mu \nu} R_g  \notag\\
T_{\mu \nu} =& \ip{\ptl_\mu U}{\ptl_\nu U}_h - \halb g_{\mu \nu} \ip{\ptl^\sigma U}{\ptl_\sigma U}_h
\end{align}

The dimensional reduction process is based on the projection of the Ricci tensor, as shown below.

\begin{subequations}
\begin{align}
\leftexp{(\bar{g})}{\text{Ric}}_{\mu \nu} =& \leftexp{(\tilde{g})}{\text{Ric}}_{\mu \nu} - \ptl_\mu \gamma \ptl_\nu \gamma - \leftexp{(g)}{\grad}_\mu   \leftexp{(g)}{\grad}_\nu \gamma - \halb 
e^{2 \gamma} \mathcal{F}_{\mu \sigma} \mathcal{F} ^{\sigma }_{\,\,\,\,\nu} \label{ricci-pure}\\
\leftexp{(\bar{g})}{\text{Ric}}_{\mu 3}=& - \halb e^{-\gamma} \leftexp{(g)}{\grad}_\sigma (e^{3 \gamma} \mathcal{F}_{\mu}^{\,\,\,\,\sigma}) \label{ricci-cross} \\
\leftexp{(\bar{g})} {\text{Ric}}_{33}=& - e^{2 \gamma} ( g^{\mu \nu}   \leftexp{(g)}{\grad}_\mu   \leftexp{(g)}{\grad}_\nu \gamma + g^{\mu \nu} \ptl_\mu \gamma  \ptl_\nu - \frac{1}{4} \mathcal{F}_{\mu \nu} \mathcal{F}^{\mu \nu}) \label{ricci-ext}
\end{align}
\end{subequations}
where, the Faraday tensor $\mathcal{F}$ is given by $\mathcal{F} = d\mathcal{A}$ i.e., $F_{\mu \nu} = \ptl_\mu \mathcal{A}_\nu - \ptl_\nu \mathcal{A}_\nu$ in the abstract index notation. The vacuum Einstein equations provide us with the equation, from \eqref{ricci-cross} above,  

\begin{align}
0=\leftexp{(\bar{g})}{\text{Ric}}_{\mu 3}=& - \halb e^{-\gamma} \leftexp{(g)}{\grad}_\sigma (e^{3 \gamma} \mathcal{F}_{\mu}^{\,\,\,\,\sigma}), \quad (\bar{M}, \bar{g})
\end{align}
which can be compactly represented as $dG =0$ where
$G \fdg = e^{3\gamma} \leftexp{*}{F}$ and by Poincare Lemma
which allows us to introduce a potential function $\omega$ such that 
\begin{align}
d \omega = G, \quad (M, g)
\end{align}

Subsequently, we shall form a conformal transformation and reduce the vacuum Einstein equations to the dimensionally reduced Einstein -wave map system. 

Actually, in general, depending on the geometry and topology of the manifold $\Sigma$ (or the Lorentzian spacetime $M$) we have a general form of the `twist' potential formalism where $G = d \omega + H,$ in which case, the wave map equation generalizes to 

\begin{align}
\square_g \omega - 4 e^{-4 \gamma} g^{\a \b} \ptl_a \gamma \ptl_b \omega + \leftexp{(g)}{\grad} _\mu ( e^{3 \gamma} H^\mu) =0, \quad (M, g)
\end{align}

 However, we shall continue and perform the dimensional reduction of Riemannian and Weyl tensors in this work. 

In the context of the initial value problem, it is relevant to reformulate the above mentioned 
geometric structures in terms of the $3+1$ and $2+1$ decompositions (i.e., the ADM decomposition). In the ADM version of the things, we have the decomposition $\olin{\Sigma} \fdg = \Sigma \times \mathbb{R}$
\begin{align}\label{ADM-red}
\bar{g} = e^{-2\gamma} \left( -N^2 dt^2 + q_{ab} (dx^a + N^a) \otimes (dx^b + N^b) \right) + e^{2 \gamma} \left(dx^2 + \mathcal{A}_0 dt + \mathcal{A}_a dx^2 \right)
\end{align}

Starting from the ADM action functional, in terms of the ADM phase space $X_{\text{ADM}} = \{ { q_{ij}, \pi^{ij}}, i, j = 1, 2, 3 \}$
\begin{align}
 \olin{S}_{\text{ADM}} \fdg = \int ( \pi^{ab} \ptl_t q_{ab} - N H - N^i  H_i )d^4 x	
\end{align}
where the constraints $H, H_i$ are given by

\begin{align}
H=& \bar{\mu}^{-1}_{\bar{q}} ( \Vert \pi \Vert^2_{\bar{q}} - \halb \text{tr} (\pi)^2) - \bar{\mu}_{\bar{q}} \bar{R}_{\bar{q}} \\
H_i=&  -2 \leftexp{(\bar{q})}{\grad}_i \pi^j_i
\end{align}
for the ADM $3+1$ decomposition of the metric $\bar{g}:$
\begin{align}
\bar{g} = -\bar{N}^2 dt^2 + \bar{q}_{ij} (dx^i + N^i dt  )\otimes (dx^j + N^j dt). 
\end{align}
Now then, in the ADM dimensionally reduced decomposition of the metric \eqref{ADM-red}
the Hamiltonian and Momentum constraints can be expressed as 

\begin{subequations}
	\begin{align}
H  \fdg =& \bar{\mu}^{-1}_q \Big ( \Vert \pi \Vert^2_q - \textnormal{tr}(\pi)^2) + \frac{1}{8} p_{\gamma}^2 + \frac{1}{2} e^{-4 \gamma} q_{ab} \mathcal{E}^a \mathcal{E}^b \Big)  \\
& + \bar{\mu}_q \Big(-R_q + 2 q^{ab} \ptl_a \gamma \ptl_b \gamma 
+ \frac{1}{4} q^{ab} q^{bd} \ptl_{[b} \mathcal{A}_{a]} \ptl_{[d} \mathcal{A}_{c]} \Big)  \notag\\
=& \bar{\mu}^{-1}_q \Big ( \Vert \pi \Vert^2_q - \textnormal{tr}(\pi)^2) + \frac{1}{8} p_{\gamma}^2 + \frac{1}{2} e^{-4 \gamma} q_{ab} \mathcal{E}^a \mathcal{E}^b \Big)  \notag\\
& + \bar{\mu}_q \Big(-R_q + 2 q^{ab} \ptl_a \gamma \ptl_b \gamma 
+ \frac{1}{4} q^{ab} q^{bd} \mathcal{F}_{ba} \mathcal{F}_{dc} \Big)  \\
H_a=& -2 \leftexp{(q)}{\grad}_b \mbo{\pi}^b_a + p \ptl_a \log \vert  \Phi \vert + \mathcal{E}^b ( \ptl_{[a} \mathcal{A}_{b]}) \\
=& -2 \leftexp{(q)}{\grad}_b \mbo{\pi}^b_a + p_\gamma \ptl_a \gamma + \mathcal{E}^b \mathcal{F}_{ba}
\end{align}
\end{subequations}

For the dimensionally reduced phase space 

\begin{align}
X^{\text{KK}} \fdg = \Big\{ (q_{ab}, \pi^{ab} ), (\gamma, p_\gamma), (\mathcal{A}_a, \mathcal{E}^a) \Big\}.
\end{align}

An important aspect of the dimensional reduction process, is that we can further transform the $\{ ( \mathcal{A}_a, \mathcal{E}^a) \}$ canonical pairs into a `twist' potential formalism. 
Consequently, we have

\begin{subequations}
\begin{align}
H =&   \bar{\mu}^{-1}_q \Big ( \Vert \pi \Vert^2_q - \textnormal{tr}(\pi)^2) + \frac{1}{8} p_{\gamma}^2 + \halb e^{2 \gamma} p^2_{\omega} \Big)  \notag\\
& + \bar{\mu}_q \Big(-R_q + 2 q^{ab} \ptl_a \gamma \ptl_b \gamma 
+ \halb e^{-4 \gamma} q^{ab} \ptl_a \omega \ptl_b \omega \Big)\\
H_a =&  -2 \leftexp{(q)}{\grad}_b \mbo{\pi}^b_a + p_\gamma \ptl_a \gamma + p_{\omega} \ptl_a \omega  
\end{align}
\end{subequations}

for the reduced Einstein-wave map phase space 

\begin{align}
X^{\text{WM}} \fdg = \Big\{ ( q_{ab}, \pi^{ab}), ( p_\gamma, \gamma), ( p_\omega, \omega) \Big\}
\end{align}

In this dimensional reduction framework, the Einstein equations can be represented as 
\begin{subequations}
\begin{align}
\ptl_t \gamma =& \frac{1}{4} N \bar{\mu}^{-1}_q p_\gamma + \mathcal{L}_N \gamma, \\
\ptl_t p_\gamma =& 4 \ptl_b \left( N \bar{\mu}_q q^{ab} \ptl_a \gamma  \right) -2 N \bar{\mu}^{-1}_q e^{4 \gamma} p^2_\omega
\notag\\
& \quad + 2N \bar{\mu}_q e^{-4 \gamma} q^{ab} \ptl_a \omega \ptl_b \omega + \mathcal{L}_N p_\gamma  \\
\ptl_t \omega=& N \bar{\mu}^{-1}_q e^{4 \gamma} p_\omega + \mathcal{L}_N \omega \\
\ptl_t p_\omega=& \ptl_b (N \bar{\mu}_q e^{-4 \gamma} q^{ab} \ptl_a \omega) + \mathcal{L}_N p_{\omega} \\
\ptl_t q_{ab} =& 2 N \bar{\mu}^{-1}_q (q_{ac} q_{bd} - q_{ab} q_{cd}) \pi^{cd} + (\mathcal{L}_N q)_{ab} \\
\ptl_t \pi^{ab} =& -2N \bar{\mu}^{-1}_q ( \pi^{ac} \pi^{bd} - q_{ab} q_{cd}) \pi^{cd} + \halb N \bar{\mu}^{-1}_q q^{ab} \left(\frac{1}{8} p^2_\gamma + \halb e^{4 \gamma} p^2_\omega \right) \notag\\
& \halb N \bar{\mu}^{-1}_q q^{ab} \left( \pi^{cd} \pi_{cd} - (\pi^c_c)^2 \right) + \bar{\mu}_q (\leftexp{(q)}{\grad}^b \leftexp{(q)}{\grad}^a N - q^{ab} \leftexp{(q)}{\grad}_b \leftexp{(q)}{\grad}^b N ) \notag\\
& N \bar{\mu}_q \left(q^{ac} q^{bd} - \halb q^{ab} q^{cd}  \right) \left( 2 \ptl_c \gamma \ptl_d \gamma + \halb e^{-4 \gamma}
\ptl_c \omega \ptl_d \omega \right)
\end{align}
\end{subequations}

These evolution equations provide the propagation equations for the constraints

\begin{subequations} \label{constriant-propagation}
\begin{align}
\ptl_t H =& \ptl_a (N^a H) + \ptl_b N q^{ab} \ptl_a H + \ptl_b (N q^{ab} H_a)  \\
\ptl_t H_a =& \ptl_b (N^b H_a ) + \ptl_a N^b H_b + H \ptl_a N 
\end{align}
\end{subequations}

The constraints and their evolution is relevant because they shall play a role in the evolution of the Weyl tensor fields in the spacetime. In particular, the Weyl components are represented as

\begin{align}
\mathcal{E}^{ij} =& \Big ( \bar{\mu}_{\bar{q}} \leftexp{(\bar{q})}{\text{Ric}}^{ij} - \bar{\mu}^{-1}_{\bar{q}}  \Big( \bar{\pi}^i_{\,\,\,\ell} \bar{\pi}^{\ell j} - \halb \bar{\pi}^{ij} \bar{\pi}^{\ell}_{\,\,\,\, \ell}   \Big) \Big) \notag\\
\mathcal{B}^{ij} =& \eps^{m \ell j} \bar{\mu}^{-1}_{\bar{q}} \Big(  \leftexp{(\bar{q})}{\grad}_\ell  \bar{\pi}^{m}_{\,\,\,\, i} -\halb \delta^i_{m} \leftexp{(\bar{q})}{\grad}_\ell \text{tr}(\bar{\pi})   \Big)
\end{align}
In particular, the constraint evolution expressions \eqref{constriant-propagation} can be used to compute the evolution equations of the $ \mathcal{E}^{ij} $ and $\mathcal{B}^{ij}$ fields, analogous to the ones in \cite{Andersson_Moncrief_2004}, but specialized to our $U(1)$ symmetric Einsteinian spacetimes. 
One of the aspects that we shall focus on in this work is whether we can transform the $\mathcal{F}$ tensor related to the `vector potential' $\mathcal{A}$ for all the components of the Riemann tensor. This is a fundamental aspect of the dimensional reduction process and has implications for the non-polarised case. 

\subsection{Coordinate Conditions}
The Einstein equations are diffeomorphism invariant and the resultant gauge freedom is manifested in terms of the choice of the lapse and shift of our metric. In this work, we shall use the constant mean curvature spatial harmonic gauge. In general, the gauge conditions of the Einstein equations, under the Cartan formalism, are reflected in the gauge conditions of the corresponding (gravitational) Yang-Mills theory, whose gauge group is the non-compact Lorentz group $SO(3,1).$

In our dmensionally reduced framework, if we make a special gauge choice on the metric $q$ such that 

\begin{align}
q = e^{2 \mbo{\nu}} h, \quad h \quad \text{is a flat metric}
\end{align}

the Hamiltonian and  momentum constraint equations can be simplified as 

\begin{align}
0=H=&\bar{\mu}^{-1}_h e^{-2 \mbo{\nu}} \left( \varpi^b_a \varpi^a_b + \frac{1}{8} p^2_\gamma + \halb e^{4 \gamma} p^2_\omega \right) - \halb \bar{\mu}_h e^{2 \mbo{\nu}} \tau^2 \notag  \\
&+2 \ptl_a (\bar{\mu}_h h^{ab} \ptl_b \mbo{\nu}) -\bar{\mu}_h \left( 2-2 h^{ab} \ptl_a \gamma \ptl_b \gamma - \halb e^{-4 \gamma} h^{ab} \ptl_a \omega \ptl_b \omega  \right) \\
0 =H_a =& \leftexp{(h)}{\grad}_b \left( \bar{\mu}_h \left(  \leftexp{(h)}{\grad}^b Y_a +  \leftexp{(h)}{\grad}_a Y^b - \delta^b_a  \leftexp{(h)}{\grad}_c Y^c \right) \right) \notag \\
& + \halb \ptl_a \tau e^{2 \mbo{\nu}} \bar{\mu}_h - \halb (p_\gamma \ptl_a \gamma + p_\omega \ptl_a \omega )
\end{align}

The Hamiltonian equations provide us the following equation for the propagation of the spacetime mean curvature: 

\begin{align}
\ptl_t \tau =& - \Delta_q N + \bar{\mu}^{-1}_q  \Big( \Vert \pi \Vert^2_q + \frac{1}{8} p^2_{\gamma} + \halb e^{2 \gamma} p^2_{\omega}  \Big)
\intertext{which can also be represented as}
\ptl_t \tau =& e^{-2 \lambda } -\Delta_h N + N \Big( e^{-4 \lambda} \sqrt{h}^{-1} \big( \Vert  \varpi \Vert^2
+  \halb \tau^2e^{4 \lambda} \sqrt{h} + \frac{1}{8} p^2_{\gamma} + \halb e^{2 \gamma} p^2_{\omega} \big) \Big)
\end{align}

Likewise, the propagation of the conformally flat form of the metric is provides an equation for the shift vector field 

\begin{align}
2N \left( \pi^{ab} - \halb q^{ab} \text{tr} (\pi) \right) + \sqrt{h} \left(\text{CK} (h, N)\right) =0
\end{align}
where $\text{CK} (h, N)$ is the conformal Killing operator 

$$\text{CK} (h, N) \fdg = \leftexp{(h)}{\grad}^a N^b + \leftexp{(h)}{\grad}^b N^a - h^{ab} \leftexp{(h)}{\grad}_a N^a. $$

\section{Cartan Formalism of  $U(1)$ Symmetric Spacetimes}

\subsection{`Polarized' $U(1)$ Einstein Equations}
In the following we shall restrict our attention to the polarized case. We shall extend these results for the generalized case later. In the polarized case, the metric has the following structure, 
\begin{align}
\bar{g} = e^{-2 \gamma} g + e^{2 \gamma} (dx^3)^2
\end{align}
Let us compute the curvature components, firstly for the general form of the dimensionally reduced Lorentz metric $g$ and then for the ADM form and the conformally flat form. 

Let us first perform the dimensional reduction of the Riemann tensor for the following decomposition of the $\bar{g}$ metric tensor
\begin{align}
\bar{g} = \tilde{g} + e^{2 \gamma} (dx^3)^2
\end{align}

It may be noted that in this decomposition, we have $ \sqrt{-\bar{g}} = e^{2 \gamma} \sqrt{-\tilde{g}}.$ Now consider the orthonormal basis of $\tilde{g}:$

\begin{align}
\mbo{e}^{\hatt{a}} = \mbo{e}^{\hatt{a}}_{\,\,\,\,\mu} dx^\mu 
\intertext{and}
\mbo{e}^{\hatt{x^3}} = e^{\gamma} dx^3
\end{align}
Now then, we have 
\begin{align}
d \mbo{e}^a =& - \leftexp{(3)}{\Theta^{\hat{a}}_{\,\,\,\,\hat{b}}} \wedge \mbo{e}^{\hat{b}} \notag\\
=&- \leftexp{(4)}{\Theta}^{\hatt{a}}_{\,\,\,\, \hatt{b}} \wedge \mbo{e}^b - \leftexp{(4)}{\Theta}^{\hatt{a}}_{\,\,\,\, x^3} \wedge \mbo{e}^{\hatt{x^3}} \\
d \mbo{e}^{\hatt{x^3}} =& \ptl_a \gamma \mbo{e}^{\hatt{a}} \wedge \mbo{e}^{\hatt{x^3}} = 
\leftexp{(4)}{\Theta}^{\hatt{x^3}}_{\,\,\,\, \hatt{a}} \wedge \mbo{e}^{\hatt{a}}
\end{align}
Consequently, the Cartan connection terms are 
\begin{align}
\leftexp{(4)}{\Theta}^{\hatt{x^3}}_{\,\,\,\, \hatt{a}} = \ptl_a \gamma \mbo{e}^{\hatt{x^3}}, \quad \leftexp{(4)}{\Theta}^{\hat{a}}_{\,\,\,\, \hat{b}} = \leftexp{(3)}{\Theta}^{\hat{a}}_{\,\,\,\, \hat{b}}
\end{align}

We are now ready to represent the Riemann tensor in terms of the decomposed metric form written above

\begin{align}
\brie^{\hatt{a}}_{\,\,\,\,\hatt{b}} =&  \rie^{\hatt{a}}_{\,\,\,\,\hatt{b}} + \ptl^a \gamma \mbo{e}^{\hatt{x^3}} \wedge \ptl_b \gamma \mbo{e}^{\hatt{x^3}} \notag \\
=& d\, \leftexp{(3)}{\Theta^{\hatt{a}}_{\,\,\,\,\hatt{b}}} + \leftexp{(3)}{\Theta^{\hatt{a}}_{\,\,\,\,\hatt{b}}} \wedge \leftexp{(3)}{\Theta^{\hatt{a}}_{\,\,\,\,\hatt{b}}} + \ptl^a \gamma \mbo{e}^{\hatt{x^3}} \wedge \ptl_b \gamma \mbo{e}^{\hatt{x^3}} 
\intertext{and the cross term}
\brie^{\hatt{x^3}}_{\,\,\,\,\hatt{a}}=& d (\ptl_a \mbo{e}^{\hatt{x}^3}) + \ptl_b \gamma \mbo{e}^{x^3} \wedge \leftexp{(3)}{\Theta^b_{\,\,\,\,a}} + \halb e^{\gamma} \mbo{e}^{x^3} \wedge \leftexp{(3)}{\Theta^b_{\,\,\,\, a}}
\end{align}

Now in the abstract index notation, 

\begin{align}
\brie^{\a}_{\,\,\,\, \b \mu \nu }  =& \rie^{\a}_{\,\,\,\, \b \mu \nu} \notag\\
\brie^{x^3}_{\,\,\,\, \a x^3 \b} =& - \leftexp{(g)}{\grad}_\a \leftexp{(g)}{\grad}_\b \gamma  -\leftexp{(g)}{\grad}_\a \gamma \leftexp{(g)}{\grad}_\b \gamma  , \quad \a, \b,\mu,\nu = 0, 1, 2. 
\end{align}
Furthermore, these curvature expressions can  be computed for the ADM form of the metric $\bar{g}$: 

\begin{align}
     \bar{g}_{\mu \nu} dx^\mu \otimes db^\nu = e^{-2\gamma} ( -N^2 dt^2 + q_{ab} (dx^a + N^a dt) (dx^b + N^b dt)) + e^{2 \gamma} (dx^3)^2.
\end{align}

\subsection{$U(1)$ Symmetric Einstein Equations with a non-vanishing `vector potential'}
The general case, wthout the polarization assumption, is a general case of the $U(1)$ symmetric spacetime. An important aspect of the problem that we shall explore is the whether there are any non-local aspects of the dimensional reduction of the Riemannian and Weyl tensors. This is due to the non-vanishing of the `vector potential' term and consequently the corresponding Faraday tensor $\mathcal{F}$ is non-vanishing. 
To start with, consider the form of the $3+1$ metric $\bar{g}$ as 

\begin{align} \label{KK-metric-classis-2}
\bar{g} = \tilde{g} + e^{2 \gamma} (dx^3 + \mathcal{A}_\nu dx^\nu )^2 
\end{align}
Correspondingly, let us consider the modified  (generalized) orthonormal basis, aligning with the decomposition presented above in \eqref{KK-metric-classis-2}: 
\begin{subequations}
	\begin{align}
	\mbo{e}^a =& e^a_\mu dx^\mu \\
	\mbo{e}^{x^3} =& e^{\gamma} (dx^3 + \mathcal{A}_\nu dx^\nu) 
	\end{align}
	\end{subequations}
compute
\begin{subequations}
	\begin{align}
	\text{d} \mbo{e}^a =&  - \leftexp{(3)}{\Theta}^{\hatt{a}}_{\,\,\,\, \hatt{b} } \wedge \mbo{e}^b
	= - \Theta^a_{\,\,\,\, b} - \Theta^a_{3} \wedge \mbo{e} ^3 \\
	\text{d} \mbo{e}^{x^3} =& \ptl_a \gamma \mbo{e}^a \wedge \mbo{e}^ {3} + e^{\gamma} dA 
	= \ptl_a \gamma \mbo{e}^{\hatt{a}} \wedge \mbo{e}^ {3} + e^{\gamma} \mathcal{F} = \ptl_a \gamma \mbo{e}^a \wedge \mbo{e}^ {3} + \halb e^{\gamma} \mathcal{F}_{ab} \mbo{e}^a \wedge \mbo{e}^b \notag\\
	=& - \Theta ^{3}_{\,\,\,\, a} \wedge \omega^a  
	\end{align}
	\end{subequations}

Now then, using the Riemann curvature in Cartan formalism, is a matrix of two-forms, can now be expressed as 

\begin{align}
&\brie^{\hatt{a}}_{\,\,\,\, \hatt{b}} = \text{d} \leftexp{(3)}{\Theta}^{\hatt{a}}_{\,\,\,\,\hatt{b}} + \halb \text{d} (e^{\gamma} \mathcal{F}^{\hatt{a}}_{\,\,\,\,\hatt{a}} e^3) + \Big(  \leftexp{(3)}{\Theta} ^{\hatt{a}}_{\,\,\,\,\hatt{c}} + \halb e^{\gamma} \mathcal{F}^{\hatt{a}}_{\,\,\,\, \hatt{c}} \mbo{e}^3 \Big) \wedge  \Big(  \leftexp{(3)}{\Theta} ^{\hatt{a}}_{\,\,\,\,\hatt{b}} + \halb e^{\gamma} \mathcal{F}^{\hatt{a}}_{\,\,\,\, \hatt{b}} \mbo{e}^3 \Big) \notag \\
&+ \Big( \ptl^a \gamma \mbo{e}^{3} + \halb e ^{\gamma} \mathcal{F}^{\hatt{a}}_{\,\,\,\, \hatt{c}} \, \mbo{e}^{\hatt{c}} \Big) \wedge \Big(  \ptl_b \gamma \mbo{e}^{3} + \halb e ^{\gamma} \mathcal{F}_{\hatt{b} \hatt{d}} \, \mbo{e}^{\hatt{a}} \Big) 
\end{align}

This expression can further be simplified as 

\begin{align}
&\text{d} \leftexp{(3)}{\Theta}^a_{\,\,\,\, b} + \halb e^{\sigma } \ptl_c \sigma \mathcal{F}^{\hatt{a}}_{\,\,\,\, \hatt{b}} \mbo{e}^{\hatt{c}} \wedge \mbo{e}^{\hatt{3}} + \halb e^{\gamma} \ptl_c \mathcal{F}^{\hatt{a}}_{\,\,\,\, \hatt{b}} \mbo{e}^{\hatt{c}} \wedge \mbo{e}^{\hatt{3}}  + \halb e^{\gamma} \mathcal{F}^a_{b} ( \ptl_b \gamma \mbo{e}^{\hatt{c}} 
\wedge \mbo{e}^{\hatt{c}} + e^\gamma \mathcal{F}) \notag\\
&+ \leftexp{(3)}{\Theta}^{\hatt{a}}_{\,\,\,\, \hatt{c}} \wedge \leftexp{(3)}{\Theta}^{\hatt{c}}_{\,\,\,\, \hatt{b}} + \halb e^{\gamma} (\mathcal{F}^{\hatt{a}}_{\,\,\,\, \hatt{c}} \mbo{e}^3 \wedge \ptl_b \gamma \mbo{e}^{\hatt{3}} + 
\leftexp{(3)}{\Theta}^{\hatt{a}}_{\,\,\,\, \hatt{c}} \wedge \mathcal{F}^{\hatt{c}}_{\,\,\,\, b} \mbo{e}^{\hatt{3}}) + \frac{1}{4} e^{2 \gamma} \mathcal{F}^{\hatt{a}}_{\,\,\,\, \hatt{c}} \mathcal{F}_{ \hatt{b} \hatt{d}} \mbo{e}^{\hatt{c}} \wedge \mbo{e}^{\hatt{d}} \notag \\
&+  \halb e^{\gamma} \Big( \mathcal{F}^{\hatt{a}}_{\,\,\,\, \hatt{c}} \mbo{e}^{\hatt{a}} \wedge \ptl_b \gamma \mbo{e}^{\hatt{3}} + \mathcal{F}_{\hatt{b} \hatt{d}} \ptl^a \gamma \mbo{e}^{\hatt{3}}  \wedge \mbo{e}^{{\hatt{d}}}\Big)
\end{align}
This can further be decomposed as 
\begin{align}
\brie^a_{\,\,\,\, bcd} = \rie^a_{\,\,\,\, bcd} + e^{2 \gamma} \Big( \mathcal{F}^{\hatt{a}}_{\,\,\,\, \hatt{c} } \mathcal{F}_{\hatt{b} \hatt{d}} - \mathcal{F}^{\hatt{a}}_{\,\,\,\, \hatt{d} } \mathcal{F}_{\hatt{b} \hatt{c}} + 2 \mathcal{F}^{\hatt{a}}_{\,\,\,\, \hatt{b} } \mathcal{F}_{\hatt{c} \hatt{d}}   \Big) 
\end{align}
Likewise, we have, the cross term: 

\begin{align}
\brie^{\hatt{x^3}}_{\,\,\,\, \hatt{a}} = \text{d} \Big( \ptl_a \gamma \mbo{e} ^{\hatt{x^3}} + \halb e^{\gamma}\mathcal{F} _{\hatt{a} \hatt{b}} \mbo{e}^{\hatt{b}} \Big) + \Big( \ptl_b \gamma \mbo{e}^{\hatt{3}} + \halb e^{\gamma} \mathcal{F}_{\hatt{b} \hatt{c}} \mbo{e}^{\hatt{c}} \Big) \wedge  \Big( \leftexp{(3)}{\Theta}^b_{\,\,\,\, \hatt{a}}+ \halb e^{\gamma} \mathcal{F}_{\hatt{b} \hatt{c}} \mbo{e}^{\hatt{c}} \Big) 
\end{align}
which can be simplified as 
\begin{align}
&\ptl^2_{ab} \gamma \mbo{e}^{\hatt{b}} \wedge \mbo{e}^{\hatt{x^3}} + \ptl_a \gamma \Big( \ptl_c \gamma \mbo{e}^{\hatt{c}} \wedge \mbo{e}^{\hatt{x^3}} + e^{\gamma} \mathcal{F} \Big)  \notag \\
&+ \halb e^{\gamma} \Big( \ptl_c \gamma \mbo{e}^{\hatt{c}} \wedge \mbo{e}^{\hatt{b}} \mathcal{F}_{\hatt{a} \hatt{b}} +   \ptl_a \mathcal{F}_{\hatt{a} \hatt{b}} \mbo{e}^{\hatt{c}} \wedge 
\mbo{e}^{\hatt{b}} - \halb \mathcal{F}_{\hatt{a} \hatt{b}} \mbo{e}^{\hatt{b}} \wedge \mbo{e}^{\hatt{c}}  \Big ) + \ptl_b \gamma \mbo{e}^{\hatt{x^3}} \wedge \leftexp{(3)}{ \Theta} ^{\hatt{b}}_{\,\,\,\, \hatt{a}} \notag \\
&+  \halb e^{\gamma} \mathcal{F}_{\hatt{b} \hatt{c}} \mbo{e}^{\hatt{c}} \wedge \leftexp{(3)}{\Theta}^{b}_{\,\,\,\, \hatt{a}} + \frac{1}{4} e^{2 \gamma} \mathcal{F}^{\hatt{b}}_{\,\,\,\, \hatt{a}} \mathcal{F}_{\hatt{b} \hatt{c}} \mbo{e}^{\hatt{c}} \wedge 
\mbo{e}^{\hatt{x^3}}
\end{align}
It now follows that we can extract the Ricci curvature (as shown in Section 2)  from

\begin{align}
\brie^{x^3}_{\,\,\,\, a x^3 b} = - \ptl^2 _{ab} \gamma - \ptl_a \gamma \ptl_a \gamma + \ptl_c 
\Gamma^c_{ba} - \frac{1}{4} e^{2 \gamma} \mathcal{F}_{ab} \mathcal{F}^c_{\,\,\,\, b}
\end{align}

and the scalar curvature: 
\begin{align}
\text{Scal} (\bar{g}) = \text{Scal} (\tilde{g}) +  \frac{1}{4} e^{2 \gamma} F^2 - 2 \square_g \sigma 
-2 ( \leftexp{(g)}{\grad} \gamma )^2 . 
\end{align}
The Riemann curvature in abstract index notation is now straightforward: 
\begin{align}
\brie^{\a}_{\,\,\,\, \b \mu \nu }  =& \rie^{\a}_{\,\,\,\, \b \mu \nu} + \frac{1}{4} ( e^{2 \gamma}) ( \mathcal{F}^{\a}_{\,\, \mu} \mathcal{F}_{\b \nu} - \mathcal{F}^{\a}_{\,\, \nu}   \mathcal{F}_{\b \mu} + 2 \mathcal{F}^{\a}_{\,\, \b} \mathcal{F}_{\mu \a} ) \notag\\
\brie^{x^3}_{\,\,\,\, \a x^3 \b} =& - \leftexp{(g)}{\grad}_\a \leftexp{(g)}{\grad}_\b \gamma  -\leftexp{(g)}{\grad}_\a \gamma \leftexp{(g)}{\grad}_\b \gamma - \frac{1}{4} e^{2 \gamma} \mathcal{F}_{\a \mu} \mathcal{F}^{\,\,\mu}_{\b}  , \quad \a, \b,\mu,\nu = 0, 1, 2. 
\end{align}
Analogously, we can compute the expressions for the ADM form of the metric $\bar{g}:$
\begin{align}
     \bar{g}_{\mu \nu} dx^\mu \otimes db^\nu =& e^{-2\gamma} ( -N^2 dt^2 + q_{ab} (dx^a + N^a dt) (dx^b + N^b dt)) \\
     &+ e^{2 \gamma} (dx^3 + \mathcal{A}_0 dt + \mathcal{A}_a dx^a)^2.
\end{align}
We have established: 
\begin{lemma}
Suppose $(\bar{M}, \bar{g})$ is a $U(1)$ symmetric lorentzian spacetime and expressible in the form 
\begin{align}
    \bar{g} = e^{-2 \gamma} g + e^{2 \gamma} (dx^3 + \mathcal{A}_\a dx^\a)^2
\end{align}
in a suitably aligned coordinate system, where $\gamma, g$ and $\mathcal{A}$ are independent of the coordinate $x^3$ corresponding to the $U(1)$ symmetry, then the Riemann curvature tensor $\brie$ and its corresponding  (gravitational) Yang-Mills curvature $F$ can be expressed independently from the undifferentiated (gravitational) `vector potential' $\mathcal{A}.$ 
\end{lemma}
An analogous result can also be obtained for $U(1)$ symmetric spacetimes satisfying the Einstein-Maxwell equations. 
\subsection{Relevant Examples}
In the following, we shall represent some geometric objects of relevant spacetimes in Cartan formalism. Let us start with the  the Kasner solution. This is perhaps the simplest non-trivial example of a translational $U(1)$ symmetric spacetime. 

\begin{subequations}
\begin{align}
    \bar{g}_{\mu \nu} dx^\mu \otimes d x^\nu =& -dt^2 + t^{2p_1} (dx^1)^2 + t^{2p_2} (dx^2)^2 + t^{2p_3} (dx^3)^2  
    \intertext{where the parameters $p_1, p_2, p_3$ are constrained by}
    p_1+ p_2+p_3=&1 \\
    p^2_1 + p^2_2 + p^2_3=&1.
\end{align}
\end{subequations}

Consider the orthonormal basis: 

\begin{align}
    \mbo{e}^{\hatt{t}} = dt, \quad \mbo{e}^{\hatt{x^1}} = t^{p_1} dx^1, \mbo{e}^{\hatt{x^2}} = t^{p_2} dx^2, \mbo{e}^{\hatt{x^3}} = t^{p_3} dx^3.
\end{align}

Then, 

\begin{align}
    d \mbo{e}^{\hatt{t}}=0, d \mbo{e}^{\hatt{x^1}} = p_1 t^{p_1-1} dt \wedge dx^1= - p_1 t^{-1} \mbo{e}^{\hatt{x^1}} \wedge \mbo{e}^{\hatt{t}}
    \intertext{analogously}
    d \mbo{e}^{\hatt{x^2}} = p_2 t^{p_2-1} dt \wedge dx^2= - p_2 t^{-1} \mbo{e}^{\hatt{x^2}} \wedge \mbo{e}^{\hatt{t}}
    \intertext{and}
     d \mbo{e}^{\hatt{x^3}} = p_3 t^{p_3-1} dt \wedge dx^3= - p_3 t^{-1} \mbo{e}^{\hatt{x^3}} \wedge \mbo{e}^{\hatt{t}}.
\end{align}
Subsequently, we can compute the Riemann curvature 2-forms using the Cartan structural equation: 
\begin{align}
    \brie^{\hatt{a}}_{\,\, \hatt{b}} = d \Theta^{\hatt{a}}_{\,\, \hatt{b}} +  \Theta^{\hatt{a}}_{\,\, \hatt{c}} \wedge \Theta^{\hatt{c}}_{\,\, \hatt{b}}
\end{align}
as follows: $\brie^{\hatt{t}}_{\,\, \hatt{x^i}} =0, i=1,2, 3.$
\begin{subequations}
\begin{align}
    \brie^{\hatt{x^1}}_{\,\, \hatt{x^2}} =& d \Theta^{\hatt{x^1}}_{\,\, \hatt{x^2}} +  \Theta^{\hatt{x^1}}_{\,\, \hatt{c}} \wedge \Theta^{\hatt{c}}_{\,\, \hatt{x^2}}=\Theta^{\hatt{x^1}}_{\,\, \hatt{c}} \wedge \Theta^{\hatt{c}}_{\,\, \hatt{x^2}} \\
    =& p_1p_2 t^{-2} \mbo{e}^{\hatt{x^1}} \wedge \mbo{e}^{\hatt{x^2}}
    \intertext{analogously}
     \brie^{\hatt{x^1}}_{\,\, \hatt{x^3}} =& \Theta^{\hatt{x^1}}_{\,\, \hatt{c}} \wedge \Theta^{\hatt{c}}_{\,\, \hatt{x^2}} \\
    =& p_1p_3 t^{-2} \mbo{e}^{\hatt{x^1}} \wedge \mbo{e}^{\hatt{x^3}}
    \intertext{and}
     \brie^{\hatt{x^2}}_{\,\, \hatt{x^3}} =& \Theta^{\hatt{x^2}}_{\,\, \hatt{c}} \wedge \Theta^{\hatt{c}}_{\,\, \hatt{x^3}} \\
    =& p_1p_2 t^{-2} \mbo{e}^{\hatt{x^2}} \wedge \mbo{e}^{\hatt{x^3}}.
\end{align}
\end{subequations}

\subsubsection{Equivariant Einstein-wave map system}
Let us consider the special case of the equivariant Einstein-wave map system. The metric is then of the form: 

\begin{align} \label{equiv-metric}
    g_{\mu \nu} dx^\mu \otimes dx^\nu  = -e^{-2\Omega (t, r)} dt^2 + e^{2 \gamma (t, r)} dr^2 + r^2 d\theta^2.  
\end{align}
We have the basis vectors $\{\mbo{e}_1, \mbo{e}_2, \mbo{e}_3 \}$ where, 

\begin{align}
    \mbo{e}^{\hatt{t}} =  e^{\Omega(t, r)} d\,t, \quad \mbo{e}^{\hatt{r}}= e^{\gamma(t, r)} d\,r, \quad \mbo{e}^{\hatt{\theta}}= r d\, \theta. 
\end{align}
Recall that, the connection 1-forms  are defined as 

\begin{align}
    \grad_X \mbo{e}_a = \Theta^b_{\,\, a} (X) \mbo{e}_b, \quad 
    \Theta^b_{\,\, a} (\ptl_j) = \Gamma^{b}_{ja}.
\end{align}

In order to compute the spin coefficients of the metric that occurs in the equivariant  Einstein-wave map system \eqref{equiv-metric}, we need to compute the exterior derivates of $\{ \mbo{e}^{\hatt{a}} \}$ i.e., Cartan's first structural equation

\begin{align}
    d\, e^{\hatt{a}} = - \Theta^{\hatt{a}}_{\,\, \hatt{b}} \wedge e^{\hatt{b}}.
\end{align}

Now then, we have 

\begin{align}
    d\, \mbo{e}^{\hatt{t}} =& \left( \ptl_t e^{\Omega (t,r)} dt + \ptl_r e^{\Omega(t,r)} dr \right) \wedge dt \notag\\
    =& e^{\Omega(t,r)} \ptl_r \Omega(t,r) dr \wedge dt = e^{-\gamma(t,r)} \ptl_r \Omega (t,r) \mbo{e}^{\hatt{r}} \wedge \mbo{e}^{\hatt{t}} 
    \intertext{likewise}
    d \, e^{\hatt{r}} =& (\ptl_t e^{\gamma(t,r)} dt + \ptl_r e^{\gamma(t,r)} dr ) \wedge dr \notag\\
    =& \ptl_r e^{\gamma(t,r)} dt \wedge dr = e^{-\Omega(t,r)} \ptl_r \gamma(t,r) \mbo{e}^{\hatt{t}} \wedge \mbo{e}^{\hatt{r}}
    \intertext{and}
    d\, e^{\hatt{\theta}} =& dr \wedge d \theta = r^{-1} e^{-\gamma(t,r)} \mbo{e}^{\hatt{r}} \wedge \mbo{e}^{\hatt{\theta}}
\end{align}
where we have used the bi-linear property of the wedge product. 
  Let us now read-off the spin coefficients. We have, 
  
  \begin{align}
      d\, \mbo{e}^{\hatt{t}} =& - \Theta^{\hatt{t}}_{\,\, \hatt{b}} \wedge \mbo{e}^{\hatt{b}} =-- \Theta^{\hatt{t}}_{\,\, \hatt{b}} \wedge \mbo{e}^{\hatt{b}} = - e^{-\gamma(t,r)} \ptl_r \Omega(t,r) \mbo{e}^{\hatt{r}} \wedge \mbo{e}^{\hatt{t}} \notag
      \intertext{thus}
      \Theta^{\hatt{t}}_{\,\, r} =& e^{-\gamma(t,r)} \ptl_r \Omega(t,r) \mbo{e}^{\hatt{t}} 
      \intertext{likewise}
      d\, \mbo{e}^{\hatt{r}} =& - \Theta^{\hatt{r}}_{\,\, \hatt{b}} \wedge \mbo{e}^{\hatt{b}} = - \Theta^{\hatt{r}}_{\,\, \hatt{t}} \wedge \mbo{e}^{\hatt{t}} = - e^{-\Omega(t,r)} \ptl_r \gamma(t,r) \mbo{e}^{\hatt{r}} \wedge \mbo{e}^{\hatt{t}} 
      \intertext{and}
      \Theta^{\hatt{r}}_{\,\, \hatt{t}} =& e^{-\Omega(t,r)} \ptl_t \gamma(t,r) \mbo{e}^{\hatt{r}}.
      \intertext{Finally}
      d\, \mbo{e}^{\hatt{\theta}} =& - \Theta^{\hatt{\theta}}_{\,\, \hatt{b}} \wedge \mbo{e}^{\hatt{b}} = - \Theta^{\hatt{\theta}}_{\,\, \hatt{r}} \wedge \mbo{e}^{\hatt{r}} = - e^{-\gamma(t,r)} r^{-1} \mbo{e}^{\hatt{\theta}} \wedge \mbo{e}^{\hatt{r}} 
      \intertext{and}
      \Theta^{\hatt{\theta}}_{\,\, \hatt{r}} =& e^{-\gamma(t,r)}r^{-1} \mbo{e}^{\hatt{\theta}}.
  \end{align}
    Let us now turn to the components of the curvature. 
    
    \begin{align}
        \rie^{\hatt{a}}_{\quad \hatt{b}} =& d\, \Theta^{\hatt{a}}_{\,\, \hatt{b}} + \Theta^{\hatt{a}}_{\,\, \hatt{c}} \wedge \Theta^{\hatt{c}}_{\,\, \hatt{b}} \\
        =& - \halb \rie^{\hatt{a}}_{\quad \hatt{b} \hatt{c} \hatt{d}} \mbo{e}^{c} \wedge \mbo{e}^d.
    \end{align}
    
The components are then 

\begin{align}
    \rie^{\hatt{t}}_{\quad \hatt{r}} =& d\, \Theta^{\hatt{t}}_{\,\, \hatt{r}} + \Theta^{\hatt{t}}_{\,\, \hatt{b}} \wedge \Theta^{\hatt{b}}_{\,\, \hatt{r}}  \notag\\
    =& \left( \ptl_t ( e^{\Omega(t,r) - \gamma(t,r)} \ptl_r \Omega(t,r)) d\, t + \ptl_r ( e^{\Omega(t,r) - \gamma(t,r)} \ptl_r \Omega(t,r) ) d\, r \right) \wedge d\, t \notag\\
    =& \ptl_r (e^{\Omega(t,r -\gamma(t,r)} \Omega(t,r)) dr \wedge dt  \\
    \rie^{\hatt{t}}_{\quad \hatt{r}} =& d\, \Theta^{\hatt{t}}_{\,\, \hatt{\theta}} + \Theta^{\hatt{t}}_{\,\, \hatt{c}} \wedge \Theta^{\hatt{c}}_{\,\, \hatt{\theta}} \notag\\
    =& e^{-\gamma(t,r)} \ptl_r \Omega(t,r) (e^{-\Omega(t,r)} dt) \wedge e^{-\gamma(t,r)} r^{-1} (r d \theta) \notag\\
    =& e^{\Omega(t,r) - 2 \gamma(t,r)} \ptl_r \Omega(t,r) dt \wedge d \theta 
    \intertext{and}
    \rie^{\hatt{r}}_{\quad \hatt{\theta}} =& d\, \Theta^{\hatt{r}}_{\,\, \hatt{\theta}} + \Theta^{\hatt{r}}_{\,\, \hatt{c}} \wedge \Theta^{\hatt{c}}_{\,\, \hatt{\theta}} \notag\\
    =& (- e^{- \gamma(t,r)}\ptl_t \gamma(t,r) dt + e^{-\gamma(t,r)} \ptl_r \gamma(t,r) dt ) \wedge d \theta \notag\\
    =& - e^{-\gamma(t,r)} \ptl_t \gamma(t,r) dt \wedge d \theta - e^{-\gamma(t,r)} \ptl_r \gamma(t,r) dr \wedge d \theta.
\end{align}
 The tensorial components of the Riemann tensor can be recovered as follows: 
   \begin{align}
       \rie =& \rie^{\hatt{a}}_{\quad \hatt{b} \hatt{c} \hatt{d}} \mbo{e}_{\hatt{a}} \mbo{e}^{\hatt{b}} \mbo{e}^{\hatt{c}} \mbo{e}^{\hatt{d}}
       \intertext{the coordinate basis can be expressed simply as}
       \mbo{e}_{\hatt{a}} =& \mbo{e}^{\mu}_{\hatt{a}} \ptl_{x^\mu} \quad \text{and} \quad  \mbo{e}^{\hatt{a}} = \mbo{e}^{\hatt{a}}_{\mu} d x^\mu.
   \end{align}
\subsubsection{Schwarzschild Spacetime}
   
   Consider a $3+1$ dimensional spherically symmetric spacetime in polar coordinates $(t,r, \theta, \phi)$:
   
   \begin{align}
       \bar{g}_{\mu \nu} dx^\mu \otimes dx^\nu = -e^{2 \a (t,r)} dt^2 +  e^{2 \b(t,r)}dr^2 + r^2 (d \theta^2 + \sin^2 \theta d \phi^2), \mu, \nu=0,1,2,3. 
   \end{align}
Now consider the orthonormal basis, 
   
   \begin{align}
       \mbo{e}^{\hatt{t}} = e^{\a(t,r)} dt, \quad \mbo{e}^{\hatt{r}}= e^{\b(t,r)} dr,\quad  \mbo{e}^{\hatt{\theta}} = r d \theta, \quad \mbo{e}^{ \hatt{\phi}}= r \sin \theta d \phi. 
   \end{align}
 Now then, let us compute,
   
   \begin{align}
       d \mbo{e}^{\hatt{t}} = e^{\a(t,r)} \ptl_r\a(t,r) dr \wedge dt =  e^{-\b(t,r)} \ptl_r \a \mbo{e}^{\hatt{r}} \wedge \mbo{e}^{\hatt{t}} = - \Theta^{\hatt{t}}_{\,\, \hatt{a}} \wedge \mbo{e}^{\hatt{a}},
   \end{align}
   thus,
   \begin{align}
       \Theta^{\hatt{t}}_{\,\, \hatt{r}} = e^{-\b(t,r)} \ptl_r \a(t,r) \mbo{e}^{\hatt{t}}
       \end{align}
   is the only non-zero component. We have 
   
   \begin{align}
       d \mbo{e}^{\hatt{r}} =& e^{\b(t,r)} \ptl_t \b(t,r) dt \wedge dr = e^{-\a(t,r)} \ptl_r \b(t,r) \mbo{e}^{\hatt{t}} \wedge \mbo{e}^{\hatt{r}}=-\Theta^{\hatt{r}}_{\,\, \hatt{b}} \wedge \mbo{e}^{\hatt{b}}
       \intertext{and}
       \Theta^{\hatt{r}}_{\,\, \hatt{t}} =& e^{-\a(t,r)} \ptl_t \b(t,r) \mbo{e}^{\hatt{r}}.
   \end{align}
   
   Analogously, we have 
   
   \begin{align}
       d \mbo{e}^{\hatt{\theta}} = dr \wedge d \theta    \\
       \Theta^{\hatt{\theta}}_{\,\,\hatt{r}} = r^{-1} e^{\b(t,r)} \mbo{e}^{\hatt{\theta}}
   \end{align}
   and 
   \begin{align}
       d \mbo{e}^{\hatt{\phi}} =& \sin \theta dr \wedge d \phi + r \cos \theta d \theta \wedge d \phi   \notag\\
       =&- r^{-1} e^{-\b} \mbo{e}^{\hatt{\phi}} \wedge \mbo{e}^{\hatt{r}} -  r^{-1} \sin^{-1} \theta \cos \theta 
       \mbo{e}^{\hatt{\phi}} \wedge \mbo{e}^{\hatt{\theta}} \notag\\
       =& - \Theta^{\hatt{\theta}}_{\,\,\hatt{a}} \wedge \mbo{e}^{\hatt{a}}
   \end{align}
   The non-vanishing spin-connection components are 
   
   \begin{align}
       \Theta^{\hatt{\phi}}_{\,\, \hatt{r}} = r^{-1} e^{-\b(t,r)} \mbo{e}^{\hatt{\phi}}, \quad \Theta^{\hatt{\phi}}_{\,\, \hatt{\theta}} = r^{-1} \sin^{-1} \cos \theta \mbo{e}^{\hatt{\phi}} 
       \end{align}
   Let us now move on curvature 2-forms. Recall 
   $$\displaystyle \brie^{\hatt{a}}_{\,\, \hatt{b}}= d \Theta^{\hatt{a}}_{\,\, \hatt{b}} + \Theta^{\hatt{a}}_{\,\, \hatt{c}} \wedge \Theta^{\hatt{c}}_{\,\, \hatt{b}}$$
   
   \begin{subequations}
   \begin{align}
       \brie^{\hatt{t}}_{\,\, \hatt{r}} =& d \Theta^{\hatt{t}}_{\,\, \hatt{r}} +  \Theta^{\hatt{t}}_{\,\, \hatt{c}} \wedge \Theta^{\hatt{c}}_{\,\, \hatt{r}} = e^{\b(t,r)} \ptl_r ( e^{\b(t,r)} \ptl_r \a(t,r)) \mbo{e}^{\hatt{r}} \wedge \mbo{e}^{\hatt{t}} \\
       \brie^{\hatt{t}}_{\,\, \hatt{\theta}} =& \Theta^{\hatt{t}}_{\,\, \hatt{r}} \wedge \Theta^{\hatt{r}}_{\,\, \hatt{\theta}} = -e^{-2 \b{t,r}} r^{-1}\ptl_r \a(t,r) \mbo{e}^{\hatt{t}} \wedge \mbo{e}^{\hatt{\theta}}  \\
       \intertext{likewise}
       \brie^{\hatt{t}}_{\,\, \hatt{\phi}} =&  \Theta^{\hatt{t}}_{\,\, \hatt{c}} \wedge \Theta^{\hatt{c}}_{\,\, \hatt{\phi}}= - \mbo{e}^{-2\b(t,r)} r^{-1} \ptl_r \a(t,r) \mbo{e}^{\hatt{t}} \wedge \mbo{e}^{\hatt{\phi}} \\
       \brie^{\hatt{\theta}}_{\,\,\hatt{r}} =& d \Theta^{\hatt{\theta}}_{\,\, \hatt{r}} + \Theta^{\hatt{\theta}}_{\,\, \hatt{c}} \wedge \Theta^{\hatt{c}}_{\,\,\hatt{r}}
       = - e^{-\b} \ptl_r \b(t,r) dt \wedge d \theta - \ptl_r \b(t,r) e^{-\b(t,r)} dr \wedge d \theta 
       \\
       \brie^{\hatt{\phi}}_{\,\, \hatt{\theta}} =&  d \Theta^{\hatt{\phi}}_{\,\, \hatt{\theta}} + \Theta^{\hatt{\phi}}_{\,\, \hatt{r}} \wedge \Theta^{\hatt{r}}_{\,\, \hatt{\theta}}   \notag\\
       =& \sin \theta d \phi \wedge d \phi + e^{-2 \b(t,r)} \sin \theta d \phi \wedge d \theta. 
   \end{align}
   \end{subequations}

   \subsection{Curvature in ADM formalism}

The components of the Weyl curvature tensor of $U(1)$ symmetric spacetimes in the Kaluza-Klein type coordinate system, in the ADM formalism and in terms of the dimensionally reduced canonical phase space variables,  can be expressed as follows. Firstly, recall that: 

\begin{align}
    \mathcal{E}^{ij} = \Big ( \bar{\mu}_q \leftexp{(\bar{q})}{\text{Ric}}^{ij} - \bar{\mu}^{-1}_{\bar{q}}  \Big( \bar{\pi}^i_{\,\,\,\ell} \bar{\pi}^\ell_{\,\,\,j} - \halb \bar{\pi}^{ij} \bar{\pi}^{\ell}_{\,\,\,\, \ell}   \Big) \Big)
\end{align}
\begin{align}\label{eq:d10}
\mathcal{E}^{ab} &= e^{3\gamma} \bar{\mu}_{q} \left (\leftexp{(q)}{R}^{ab} - {}^{(q)} \nabla^b {}^{(q)} \nabla^a \gamma + q^{ab} \leftexp{(q)}{\grad}_c  \leftexp{(q)}{\grad}^c \gamma - 3 \leftexp{(q)}{\grad}^a \gamma  \leftexp{(q)}{\grad}^b \gamma \right. \notag \\
& + q^{ab} \leftexp{(q)}{\grad}_c \gamma  \leftexp{(q)}{\grad}^c \gamma  
 \left.{}- \frac{1}{2} e^{4\gamma} q^{ac} q^{df} q^{be} \left(\ptl_e \mathcal{A}_{d} - \ptl_d \mathcal{A}_{e}\right) \left(\ptl_c \mathcal{A}_{f} - \ptl_f \mathcal{A}_{c}\right)\right ) \notag\\
& {}- e^\gamma \bar{\mu}^{-1}_q \left(-\frac{1}{2} e^{2\gamma} \pi^{ab} \left(\frac{1}{2} p + 2q_{cd} \pi^{cd}\right) 
  {}+ e^{2\gamma} q_{cd} \pi^{ad} \pi^{bc} + \frac{1}{4} e^{-2\gamma} \mathcal{E}^a \mathcal{E}^b\right), \\
  \mathcal{E}_3^{\hphantom{3}a} &= e^{5\gamma} \bar{\mu}_q \left (\frac{1}{2} {}^{(q)} \nabla_b \left (q^{bd} q^{ac} \left(\ptl_c \mathcal{A}_{d} - \ptl_c \mathcal{A}_{c}\right)\right ) - \frac{5}{2} {}^{(q)} \nabla^b \gamma  q^{ac} \left(\ptl_c \mathcal{A}_{b} - \ptl_b \mathcal{A}_{c}\right)\right ) \notag \\
& \quad - e^\gamma \bar{\mu}^{-1}_q \left (\frac{1}{2} \mathcal{E}^c\; q_{bc} \pi^{ab} + \frac{1}{8} p_\gamma \; \mathcal{E}^a\right ) \\
\mathcal{E}_{33} &= -e^{\gamma} \bar{\mu}^{-1}_q \left (\frac{1}{4} e^{-2\gamma} q_{ab} \mathcal{E}^a \mathcal{E}^b + \frac{e^{2\gamma}}{4} p_{\gamma} \left(\frac{1}{2} p_\gamma + q_{ab} \pi^{ab}\right)\right ) \notag\\
& {}- e^{3\gamma} \left ( \vphantom{\frac{1}{4}}\partial_a \left(\bar{\mu}_q q^{ab} \ptl_b \gamma \right) + \bar{\mu}_q q^{ab} \ptl_a \gamma \ptl_b \gamma \right. \notag\\
& \left. \quad + \frac{1}{4} e^{4\gamma} \bar{\mu}_q q^{ac} q^{bd} \left( \ptl_b \mathcal{A}_{c} - \ptl_c \mathcal{A}_{b}\right) \left(\ptl_a \mathcal{A}_{d} - \ptl_d \mathcal{A}_{a}\right)\right ). 
\end{align}
Now the recall the `magnetic' components of the Weyl tensor $B^{ij}$ 

\begin{align}
    B^{ij} = \eps^{m \ell j} \bar{\mu}^{-1}_q \Big(  \leftexp{(q)}{\grad}_\ell  \bar{\pi}^{m}_{\,\,\,\, i} -\halb \delta^i_{m} \leftexp{(q)}{\grad}_\ell \text{tr}(\bar{\pi})   \Big)
\end{align}
specifically, 
\begin{align}\label{eq:d13}
 \mathcal{B}^{ab} &= \epsilon^{bc3} \Bigg (\frac{1}{2} e^{5\gamma} q_{ce} \bar{\mu}^{-1}_q \pi^{de} q^{af} \epsilon_{fd} \left(\epsilon^{mn} \ptl_n \mathcal{A}_{m}\right)
- \mathcal{E}^d \bar{\mu}^{-1}_q e^\gamma \ptl_f \gamma q^{af}    q_{dc} + \halb e^{\gamma} \bar{\mu}^{-1}_q \mathcal{E}^d \ptl_d \gamma \delta_c^a  \notag\\
 & {}- \frac{1}{2}  e^{5 \gamma} \bar{\mu}^{-1}_q \left(\frac{1}{2} p_\gamma  + q_{mn} \pi^{mn}\right) q^{af} \epsilon_{fc} \left(\epsilon^{rs} \ptl_s \mathcal{A}_{r}\right) - \frac{1}{2} e^\gamma\; {}^{(q)} \nabla_c \left( \mathcal{E}^a \bar{\mu}^{-1}_q\right) \Bigg)  , \\
\mathcal{B}_3^{\hphantom{3}a} &= \epsilon^{3ca} \Bigg( - e^{3\gamma} \bar{\mu}^{-1}_q \ptl_b \gamma \left(\pi_c^b - \delta_c^b q_{mn} \pi^{mn}\right) - \frac{1}{4} e^{3\gamma} \bar{\mu}^{-1}_q \mathcal{E}^b \left( \ptl_c \mathcal{A}_{b} - \ptl_b \mathcal{A}_{c}\right) \notag\\
& \quad + \frac{1}{4} \ptl_c( e^{3 \gamma} p_\gamma \bar{\mu}^{-1}_q)\Bigg)  . \\
\mathcal{B}_{33} &= \epsilon^{ac3} \Bigg(  e^{2\gamma} \ptl_c(\frac{1}{2} e^{-\gamma} \bar{\mu}^{-1}_q \mathcal{E}_a) - \frac{1}{2} e^{\gamma} \mathcal{E}_a \ptl_c \gamma \notag\\
&+ \halb e^{5 \gamma} \bar{\mu}^{-1}_q \left(\frac{1}{2} p_\gamma + q_{bd} \pi^{bd}\right) \left(\ptl_c \mathcal{A}_{a} - \ptl_a \mathcal{A}_{c}\right) 
  -\halb e^{5 \gamma} \bar{\mu}^{-1}_q \pi_a^f \left(\ptl_c \mathcal{A}_{f} - \ptl_f \mathcal{A}_{c}\right)\Bigg). 
\end{align}
In these formulas indices \(a, b, c, \ldots\) are raised and lowered with the Riemannian 2-metric \({}^{(2)}\tilde{g} = \tilde{g}_{ab} dx^a \otimes dx^b\), \({}^{(2)}\tilde{\nabla}_a\) designates covariant differentiations with respect to this metric whereas \(\mu_{{}^{(2)}\tilde{g}}\) and \({}^{(2)}\tilde{R}^{ab}\) are its `volume' element and Ricci tensor.

We can use these formulas, in connection with \eqref{constriant-propagation}, to compute the explicit evolution formulas of $\mathcal{E}$ and $\mathcal{B}$ i.e., $\displaystyle \frac{\ptl}{ \ptl t} \mathcal{E}$ and $\displaystyle \frac{\ptl}{ \ptl t} \mathcal{B}.$ These constructions are closely related to the constraint evolution equations that we discussed earlier. It may be noted that the Bel-Robinson energy can be constructed from the quantities $\mathcal{E}^{ij} \mathcal{B}_{ij}$ and $\mathcal{B}^{ij} \mathcal{B}_{ij}.$ Likewise, we can use these formulas to construct the explicit expressions for the ADM components of the Weyl tensor for $U(1)$ symmetric spacetimes such as the Kasner solution, equivariant $U(1)$ spacetimes and also axisymmetric spacetimes such as the Kerr black hole spacetimes.   

\bibliographystyle{plain}

	\end{document}